\documentclass[oneside]{amsart}
\usepackage[utf8]{inputenc}
\usepackage{amsmath,amssymb,amsthm, graphicx,float,url}
\usepackage[dvipsnames]{xcolor}
\usepackage{tikz-cd} 
\usepackage{caption}

\usepackage{multicol}
\newcommand{\gen}[1]{\left< \mathinner{#1} \right>}

\newcommand{\ggen}[1]{\left<\left< \mathinner{#1} \right>\right>}

\usepackage{subcaption}
\numberwithin{equation}{section}

\usepackage{dirtytalk}
 \usepackage[capitalise]{cleveref}
\usepackage{comment}

\usepackage[a4paper]{geometry}

\usepackage{tikz}
\usetikzlibrary{calc,decorations.pathmorphing,patterns,arrows,decorations.markings,positioning}
\usetikzlibrary{automata,chains,fit,shapes}

\newtheorem{theoremx}{Theorem}

\newcommand{\oldV}{V_{\text{old}}}
\newcommand{\newV}{V_{\text{new}}}
\newcommand{\GamSlex}[1]{#1-slex}

 \renewcommand{\leq}{\leqslant}

\newcommand{\bi}{\begin{itemize}}
\newcommand{\ei}{\end{itemize}}
\newcommand{\be}{\begin{enumerate}}
\newcommand{\ee}{\end{enumerate}}

\newtheorem{theorem}{Theorem}[section]
\newtheorem{proposition}[theorem]{Proposition}

\newtheorem{lemma}[theorem]{Lemma}

\theoremstyle{definition}
\newtheorem{definition}[theorem]{Definition}
\newtheorem{remark}[theorem]{Remark}
\newtheorem{example}[theorem]{Example}

\begin{document}

\title[New constructions of free products and  geodetic Cayley graphs]{New constructions of free products \\
and geodetic Cayley graphs}\thanks{Research supported by Australian Research Council grant DP210100271.}

\author[J. Abraham]{Joshua Abraham}\address{Department of Mathematics, Statistics, and Computer Science, University of Illinois Chicago, Chicago, IL 60607-7045, USA}\email{jjerr@uic.edu}

\author[M. Elder]{Murray Elder}\address{School of Mathematical and Physical Sciences, University of Technology Sydney, Broadway NSW 2007, Australia}\email{murray.elder@uts.edu.au}

\author[A. Piggott]{Adam Piggott}\address{Mathematical Sciences Institute, Australian National University, Canberra ACT  2601, Australia}
\email{adam.piggott@anu.edu.au}

\author[K. Townsend]{Kane Townsend}\address{School of Mathematical and Physical Sciences, University of Technology Sydney, Broadway NSW 2007, Australia}\email{kane.townsend@uts.edu.au}

\date{\today}

	%05C12  	Distance in graphs
%05C25  	Graphs and abstract algebra (groups, rings, fields, etc.) [See also 20F65]
%	05C76  	Graph operations (line graphs, products, etc.)
%20F65 Geometric group theory [See also 05C25, 20E08, 57Mxx]
%	20E05  	Free nonabelian groups
%20E06  	Free products of groups, free products with amalgamation, Higman-Neumann-Neumann extensions, and generalizations

\begin{abstract}A connected graph is called \emph{geodetic} if there is a unique shortest path between each pair of vertices.
We introduce a systematic method for constructing new presentations of free products that give rise to previously unknown geodetic Cayley graphs. Our approach adapts subdivision techniques of Parthasarathy and Srinivasan (J. Combin. Theory Ser. B, 1982), which preserve geodecity at the graph level, to the setting of group presentations and rewriting systems. 
Specifically, given a group $G$ with geodetic Cayley graph with respect to generating set $\Sigma$ and an integer $n$, our construction produces a rewriting system presenting 
the free product of $G$ with a free group of rank $n|\Sigma|$
with geodetic Cayley graph with respect to  
a
new generating set. This framework provides new infinite families of geodetic Cayley graphs and extends the toolkit for investigating long-standing conjectures on geodetic groups.

\end{abstract}

\maketitle

\section{Introduction}
 A connected graph is called \emph{geodetic} if there is a unique shortest path between each pair of vertices.
 The study of geodetic graphs is motivated by applications in network design and parallel computing, where uniqueness of shortest paths in a graph provides efficiency in routing algorithms (see for example \cite{Leighton1992,YANG200773}) and non-trivial 2-blocks allow for localised fault-tolerance. 
Algebraically, a 
 finitely generated group is called \emph{geodetic} if it admits a finite inverse-closed generating set  $\Sigma$ for which the associated Cayley graph is geodetic, in which case we call $\Sigma$  a \emph{geodetic generating set}.

If a group has a geodetic generating set, then it admits a  confluent length-reducing rewriting system with respect to that generating set.
Madlener and Otto \cite{MadlenerOttoLengthReducing} conjectured that a group admits a finite confluent length-reducing rewriting system if and only if it is plain (that is, a free product of finitely many finite groups and a finitely generated free group). Ten years later Shapiro \cite{shapiro1997pascal} conjectured that a finitely generated group is geodetic if and only if it is plain. Recent progress by 
several authors 
\cite{EisenbergP19,ElderP22,ElderP23,Elder07052025,AdamMonadic} has attempted to clarify the 
connection between rewriting systems and geodetic Cayley graphs, with a goal to advance both conjectures.

Federici \cite{wrap108882} conjectured that the only finite geodetic Cayley graphs are complete graphs and, in the special case of cyclic groups of odd order, an odd cycle.
    This conjecture has been confirmed in some special infinite families, and  checked by computer for all groups up to order $1024$ in \cite[Theorems B and C]{Elder07052025}. 

In this paper we investigate systematic methods for constructing new geodetic Cayley graphs from existing ones. At the graph-theoretic level, it is known that if $\Gamma$ is a geodetic graph then replacing every edge of $\Gamma$ by a path of fixed odd length yields another geodetic graph \cite{parthasarathy1982some}.
Inspired by this, we introduce a construction that takes as input a group $G$ with finite generating inverse-closed set $\Sigma$ and outputs a rewriting system $\nabla_n(G,\Sigma)$ corresponding to a subdivision of the Cayley graph. We show that $\nabla_n(G,\Sigma)$ presents 
the free product of $G$ with a free group of rank $n|\Sigma|$,  which we denote by $G \ast F_{n|\Sigma|}$,
and that geodeticity of the generating set is preserved under this transformation.

Our main results are the following. First, we prove that the construction $\nabla_n(G,\Sigma)$ yields a confluent, inverse-closed rewriting system that presents $G \ast F_{n|\Sigma|}$ (\cref{thm:construction}). Second, we prove that $\Sigma$ is a geodetic generating set for $G$ if and only if the expanded generating set $\Sigma(n)$ is geodetic for the constructed free product (\cref{thm:constructiongeodeticiffgeodetic}). This provides a general framework for building new families of geodetic Cayley graphs 
from known ones.

\subsection{Prelude: known ways to construct geodetic Cayley graphs}

See \cref{subsec:grouppresentations} for definitions of group presentation, Cayley graph and geodetic generating set used in this section.

\begin{example}\label{eg:shortcutinfreegroups}
Let $X$ be a finite set with $|X|=k$. We denote the \emph{free group on $X$} by $F(X)$, with $X$ being a \emph{free basis} for $F(X)$. The \emph{rank} of $F(X)$ is $k$. The group presentation of $F(X)$ is given by $\langle X\mid -\rangle$. We also denote $F(X)$ by $F_k$ when we do not wish to explicitly reference the free basis $X$. A free group is a geodetic group: $F(X)$ has geodetic generating set given by $\Sigma=X\cup X^{-1}$. Other geodetic generating sets exist. Let $X=\{x,y\}$, then both $\Sigma=\{\textcolor{red}{x}^{\pm 1},\textcolor{blue}{y}^{\pm 1}\}$ and $\Omega=\{\textcolor{red}{x}^{\pm 1},\textcolor{blue}{y}^{\pm 1},(\textcolor{ForestGreen}{xy})^{\pm 1}\}$ are geodetic generating sets for $F(X)$. See a comparison of the Cayley graphs of $F(X)$ with generating sets $\Sigma$ and $\Omega$  in \cref{fig:f2}.

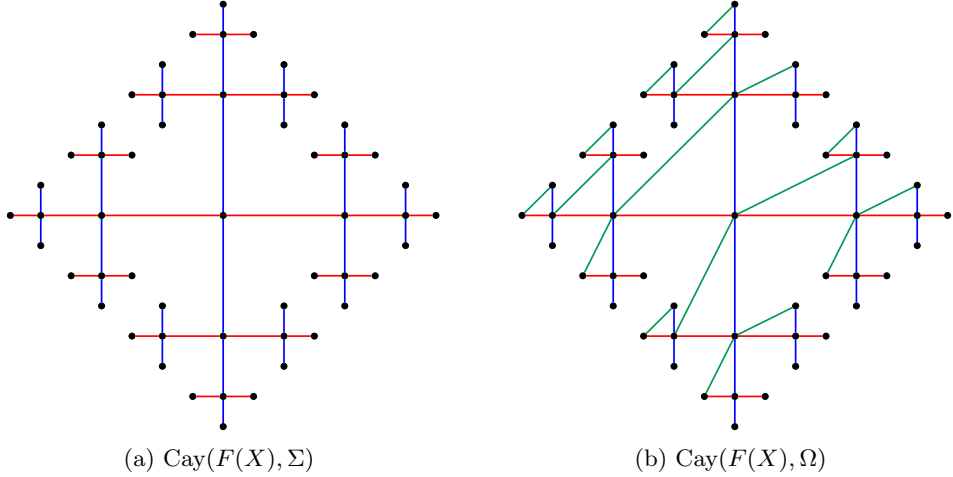
\begin{figure}[ht]
\begin{center}

 \begin{subfigure}{0.45\textwidth}
 \begin{center}
	\scalebox{0.8}{\begin{tikzpicture}

%radius 1
	\node[circle, draw, fill=black,inner sep=1pt] (0_0) at (0,0) {};
        	\node[circle, draw, fill=black,inner sep=1pt] (2_0) at (2,0) {};
         	\node[circle, draw, fill=black,inner sep=1pt] (0_2) at (0,2) {};
         	\node[circle, draw, fill=black,inner sep=1pt] (-2_0) at (-2,0) {};
         	\node[circle, draw, fill=black,inner sep=1pt] (0_-2) at (0,-2) {};
\draw[thick,red] (0_0) --  (2_0);
\draw[thick,red] (-2_0) --  (0_0);
\draw[thick,blue] (0_0) --  (0_2);
\draw[thick,blue] (0_-2) --  (0_0);

%radius 2
 	\node[circle, draw, fill=black,inner sep=1pt]  (2_-1) at (2,-1) {};
	\node[circle, draw, fill=black,inner sep=1pt] (2_1) at (2,1) {};
     	\node[circle, draw, fill=black,inner sep=1pt] (3_0) at (3,0) {};
 \draw[thick,red] (2_0) --  (3_0);
\draw[thick,blue] (2_0) --  (2_1);
\draw[thick,blue] (2_-1) --  (2_0);
        
	 \node[circle, draw, fill=black,inner sep=1pt] (-2_-1) at (-2,-1) {};
      	\node[circle, draw, fill=black,inner sep=1pt]  (-2_1) at (-2,1) {};
    \node[circle, draw, fill=black,inner sep=1pt]  (-3_0) at (-3,0) {};
        \draw[thick,red] (-3_0) --  (-2_0);
\draw[thick,blue] (-2_0) --  (-2_1);
\draw[thick,blue] (-2_-1) --  (-2_0);
 
   	\node[circle, draw, fill=black,inner sep=1pt]  (-1_2) at (-1,2) {};
      	\node[circle, draw, fill=black,inner sep=1pt] (1_2) at (1,2) {};
           	\node[circle, draw, fill=black,inner sep=1pt] (0_3) at (0,3) {};
     \draw[thick,blue] (0_2) --  (0_3);
\draw[thick,red] (0_2) --  (1_2);
\draw[thick,red] (-1_2) --  (0_2);
     
\node[circle, draw, fill=black,inner sep=1pt] (-1_-2) at (-1,-2) {};
     	\node[circle, draw, fill=black,inner sep=1pt]	 (1_-2) at (1,-2) {};
  	\node[circle, draw, fill=black,inner sep=1pt] (0_-3) at (0,-3) {};
    \draw[thick,blue] (0_-3) --  (0_-2);
\draw[thick,red] (0_-2) --  (1_-2);
\draw[thick,red] (-1_-2) --  (0_-2);

%radius 3
    \node[circle, draw, fill=black,inner sep=1pt] (3_-0point5) at (3,-0.5) {};
        \node[circle, draw, fill=black,inner sep=1pt] (3_0point5) at (3,0.5) {};
     \node[circle, draw, fill=black,inner sep=1pt] (3point_0) at (3.5,0) {};
     
         \draw[thick,red] (3_0) --  (3point_0);
\draw[thick,blue] (3_0) --  (3_0point5);
\draw[thick,blue] (3_-0point5) --  (3_0);
    
\node[circle, draw, fill=black,inner sep=1pt] (-3_-0point5) at (-3,-0.5) {};
    \node[circle, draw, fill=black,inner sep=1pt] (-3_0point5) at (-3,0.5) {};
\node[circle, draw, fill=black,inner sep=1pt] (-3point5_0) at (-3.5,0) {};

\draw[thick,red] (-3point5_0) --  (-3_0);
\draw[thick,blue] (-3_0) --  (-3_0point5);
\draw[thick,blue] (-3_-0point5) --  (-3_0);

 \node[circle, draw, fill=black,inner sep=1pt] (1point5_2) at (1.5,2) {};        \node[circle, draw, fill=black,inner sep=1pt]	(1_2point5) at (1,2.5) {};
     \node[circle, draw, fill=black,inner sep=1pt] (1_1point5) at (1,1.5) {};
    
      \draw[thick,red] (1_2) --  (1point5_2);
\draw[thick,blue] (1_2) --  (1_2point5);
\draw[thick,blue] (1_1point5) -- (1_2);
        	
\node[circle, draw, fill=black,inner sep=1pt]	 (1point5_-2) at (1.5,-2) {};
 \node[circle, draw, fill=black,inner sep=1pt] (1_-2point5) at (1,-2.5) {};
 \node[circle, draw, fill=black,inner sep=1pt]	(1_-1point5) at (1,-1.5) {};
       \draw[thick,red] (1_-2) --  (1point5_-2);
\draw[thick,blue] (1_-2point5) --  (1_-2);
\draw[thick,blue] (1_-2) -- (1_-1point5);
	
\node[circle, draw, fill=black,inner sep=1pt]  (-1point5_-2) at (-1.5,-2) {};
\node[circle, draw, fill=black,inner sep=1pt] (-1_-2point5) at (-1,-2.5) {};
 \node[circle, draw, fill=black,inner sep=1pt] (-1_-1point5) at (-1,-1.5) {};
     \draw[thick,red] (-1point5_-2) --  (-1_-2);
\draw[thick,blue] (-1_-2point5) --  (-1_-2);
\draw[thick,blue] (-1_-2) -- (-1_-1point5);

\node[circle, draw, fill=black,inner sep=1pt] (-1point5_2) at (-1.5,2) {};
\node[circle, draw, fill=black,inner sep=1pt] (-1_2point5) at (-1,2.5) {};
\node[circle, draw, fill=black,inner sep=1pt] (-1_1point5) at (-1,1.5) {};
 \draw[thick,red] (-1point5_2) --  (-1_2);
\draw[thick,blue] (-1_2point5) --  (-1_2);
\draw[thick,blue] (-1_2) -- (-1_1point5);

\node[circle, draw, fill=black,inner sep=1pt]	 (0_-3point5) at (0,-3.5) {};
\node[circle, draw, fill=black,inner sep=1pt] (0point5_-3) at (0.5,-3) {};
\node[circle, draw, fill=black,inner sep=1pt] (-0point5_-3) at (-0.5,-3) {};
\draw[thick,blue] (0_-3point5) --  (0_-3);
\draw[thick,red] (0_-3) --  (0point5_-3);
\draw[thick,red] (-0point5_-3) --  (0_-3);

\node[circle, draw, fill=black,inner sep=1pt] (0_3point5) at (0,3.5) {};
\node[circle, draw, fill=black,inner sep=1pt] (0point5_3) at (0.5,3) {};
\node[circle, draw, fill=black,inner sep=1pt] (-0point5_3) at (-0.5,3) {};
          \draw[thick,blue] (0_3point5) --  (0_3);
\draw[thick,red] (0_3) --  (0point5_3);
\draw[thick,red] (-0point5_3) --  (0_3);

\node[circle, draw, fill=black,inner sep=1pt]	 (1point5_1) at (1.5,1) {};
\node[circle, draw, fill=black,inner sep=1pt]	 (2point5_1) at (2.5,1) {};
\node[circle, draw, fill=black,inner sep=1pt] (2_1point5) at (2,1.5) {};
      \draw[thick,blue] (2_1) --  (2_1point5);
\draw[thick,red] (2_1) --  (2point5_1);
\draw[thick,red] (1point5_1) --  (2_1);

\node[circle, draw, fill=black,inner sep=1pt] (-1point5_1) at (-1.5,1) {};
\node[circle, draw, fill=black,inner sep=1pt] (-2point5_1) at (-2.5,1) {};
\node[circle, draw, fill=black,inner sep=1pt] (-2_1point5) at (-2,1.5) {};
 \draw[thick,blue] (-2_1) --  (-2_1point5);
\draw[thick,red] (-2point5_1) --  (-2_1);
\draw[thick,red] (-2_1) --  (-1point5_1);

\node[circle, draw, fill=black,inner sep=1pt]  (2_-1point5) at (2,-1.5) {};
\node[circle, draw, fill=black,inner sep=1pt] (2point5_-1) at (2.5,-1) {};
\node[circle, draw, fill=black,inner sep=1pt] (1point5_-1) at (1.5,-1) {};
          \draw[thick,blue] (2_-1) --  (2_-1point5);
\draw[thick,red] (2_-1) --  (2point5_-1);
\draw[thick,red] (1point5_-1) --  (2_-1);

\node[circle, draw, fill=black,inner sep=1pt] (-2_-1point5) at (-2,-1.5) {};
\node[circle, draw, fill=black,inner sep=1pt] (-2point5_-1) at (-2.5,-1) {};
\node[circle, draw, fill=black,inner sep=1pt] (-1point5_-1) at (-1.5,-1) {};
          \draw[thick,blue] (-2_-1) --  (-2_-1point5);
\draw[thick,red] (-2_-1) --  (-2point5_-1);
\draw[thick,red] (-1point5_-1) --  (-2_-1);

			\end{tikzpicture}}
\caption{$\text{Cay}(F(X),\Sigma)$
\label{fig:f2freebasis}}
\end{center}
\end{subfigure}
      \begin{subfigure}{0.45\textwidth}
      \begin{center}
\scalebox{0.8}{\begin{tikzpicture}

	\node[circle, draw, fill=black,inner sep=1pt] (0_0) at (0,0) {};
        	\node[circle, draw, fill=black,inner sep=1pt] (2_0) at (2,0) {};
         	\node[circle, draw, fill=black,inner sep=1pt] (0_2) at (0,2) {};
         	\node[circle, draw, fill=black,inner sep=1pt] (-2_0) at (-2,0) {};
         	\node[circle, draw, fill=black,inner sep=1pt] (0_-2) at (0,-2) {};
\draw[thick,red] (0_0) --  (2_0);
\draw[thick,red] (-2_0) --  (0_0);
\draw[thick,blue] (0_0) --  (0_2);
\draw[thick,blue] (0_-2) --  (0_0);

%radius 2
 	\node[circle, draw, fill=black,inner sep=1pt]  (2_-1) at (2,-1) {};
	\node[circle, draw, fill=black,inner sep=1pt] (2_1) at (2,1) {};
     	\node[circle, draw, fill=black,inner sep=1pt] (3_0) at (3,0) {};
 \draw[thick,red] (2_0) --  (3_0);
\draw[thick,blue] (2_0) --  (2_1);
\draw[thick,blue] (2_-1) --  (2_0);
        
	 \node[circle, draw, fill=black,inner sep=1pt] (-2_-1) at (-2,-1) {};
      	\node[circle, draw, fill=black,inner sep=1pt]  (-2_1) at (-2,1) {};
    \node[circle, draw, fill=black,inner sep=1pt]  (-3_0) at (-3,0) {};
        \draw[thick,red] (-3_0) --  (-2_0);
\draw[thick,blue] (-2_0) --  (-2_1);
\draw[thick,blue] (-2_-1) --  (-2_0);
 
   	\node[circle, draw, fill=black,inner sep=1pt]  (-1_2) at (-1,2) {};
      	\node[circle, draw, fill=black,inner sep=1pt] (1_2) at (1,2) {};
           	\node[circle, draw, fill=black,inner sep=1pt] (0_3) at (0,3) {};
     \draw[thick,blue] (0_2) --  (0_3);
\draw[thick,red] (0_2) --  (1_2);
\draw[thick,red] (-1_2) --  (0_2);
     
\node[circle, draw, fill=black,inner sep=1pt] (-1_-2) at (-1_-2) {};
     	\node[circle, draw, fill=black,inner sep=1pt]	 (1_-2) at (1,-2) {};
  	\node[circle, draw, fill=black,inner sep=1pt] (0_-3) at (0,-3) {};
    \draw[thick,blue] (0_-3) --  (0_-2);
\draw[thick,red] (0_-2) --  (1_-2);
\draw[thick,red] (-1_-2) --  (0_-2);

%radius 3
    \node[circle, draw, fill=black,inner sep=1pt] (3_-0point5) at (3,-0.5) {};
        \node[circle, draw, fill=black,inner sep=1pt] (3_0point5) at (3,0.5) {};
     \node[circle, draw, fill=black,inner sep=1pt] (3point_0) at (3.5,0) {};
     
         \draw[thick,red] (3_0) --  (3point_0);
\draw[thick,blue] (3_0) --  (3_0point5);
\draw[thick,blue] (3_-0point5) --  (3_0);
    
\node[circle, draw, fill=black,inner sep=1pt] (-3_-0point5) at (-3,-0.5) {};
    \node[circle, draw, fill=black,inner sep=1pt] (-3_0point5) at (-3,0.5) {};
\node[circle, draw, fill=black,inner sep=1pt] (-3point5_0) at (-3.5,0) {};

\draw[thick,red] (-3point5_0) --  (-3_0);
\draw[thick,blue] (-3_0) --  (-3_0point5);
\draw[thick,blue] (-3_-0point5) --  (-3_0);

 \node[circle, draw, fill=black,inner sep=1pt] (1point5_2) at (1.5,2) {};        \node[circle, draw, fill=black,inner sep=1pt]	(1_2point5) at (1,2.5) {};
     \node[circle, draw, fill=black,inner sep=1pt] (1_1point5) at (1,1.5) {};
    
      \draw[thick,red] (1_2) --  (1point5_2);
\draw[thick,blue] (1_2) --  (1_2point5);
\draw[thick,blue] (1_1point5) -- (1_2);
        	
\node[circle, draw, fill=black,inner sep=1pt]	 (1point5_-2) at (1.5,-2) {};
 \node[circle, draw, fill=black,inner sep=1pt] (1_-2point5) at (1,-2.5) {};
 \node[circle, draw, fill=black,inner sep=1pt]	(1_-1point5) at (1,-1.5) {};
       \draw[thick,red] (1_-2) --  (1point5_-2);
\draw[thick,blue] (1_-2point5) --  (1_-2);
\draw[thick,blue] (1_-2) -- (1_-1point5);
	
\node[circle, draw, fill=black,inner sep=1pt]  (-1point5_-2) at (-1.5,-2) {};
\node[circle, draw, fill=black,inner sep=1pt] (-1_-2point5) at (-1,-2.5) {};
 \node[circle, draw, fill=black,inner sep=1pt] (-1_-1point5) at (-1,-1.5) {};
     \draw[thick,red] (-1point5_-2) --  (-1_-2);
\draw[thick,blue] (-1_-2point5) --  (-1_-2);
\draw[thick,blue] (-1_-2) -- (-1_-1point5);

\node[circle, draw, fill=black,inner sep=1pt] (-1point5_2) at (-1.5,2) {};
\node[circle, draw, fill=black,inner sep=1pt] (-1_2point5) at (-1,2.5) {};
\node[circle, draw, fill=black,inner sep=1pt] (-1_1point5) at (-1,1.5) {};
 \draw[thick,red] (-1point5_2) --  (-1_2);
\draw[thick,blue] (-1_2point5) --  (-1_2);
\draw[thick,blue] (-1_2) -- (-1_1point5);

\node[circle, draw, fill=black,inner sep=1pt]	 (0_-3point5) at (0,-3.5) {};
\node[circle, draw, fill=black,inner sep=1pt] (0point5_-3) at (0.5,-3) {};
\node[circle, draw, fill=black,inner sep=1pt] (-0point5_-3) at (-0.5,-3) {};
\draw[thick,blue] (0_-3point5) --  (0_-3);
\draw[thick,red] (0_-3) --  (0point5_-3);
\draw[thick,red] (-0point5_-3) --  (0_-3);

\node[circle, draw, fill=black,inner sep=1pt] (0_3point5) at (0,3.5) {};
\node[circle, draw, fill=black,inner sep=1pt] (0point5_3) at (0.5,3) {};
\node[circle, draw, fill=black,inner sep=1pt] (-0point5_3) at (-0.5,3) {};
          \draw[thick,blue] (0_3point5) --  (0_3);
\draw[thick,red] (0_3) --  (0point5_3);
\draw[thick,red] (-0point5_3) --  (0_3);

\node[circle, draw, fill=black,inner sep=1pt]	 (1point5_1) at (1.5,1) {};
\node[circle, draw, fill=black,inner sep=1pt]	 (2point5_1) at (2.5,1) {};
\node[circle, draw, fill=black,inner sep=1pt] (2_1point5) at (2,1.5) {};
      \draw[thick,blue] (2_1) --  (2_1point5);
\draw[thick,red] (2_1) --  (2point5_1);
\draw[thick,red] (1point5_1) --  (2_1);

\node[circle, draw, fill=black,inner sep=1pt] (-1point5_1) at (-1.5,1) {};
\node[circle, draw, fill=black,inner sep=1pt] (-2point5_1) at (-2.5,1) {};
\node[circle, draw, fill=black,inner sep=1pt] (-2_1point5) at (-2,1.5) {};
 \draw[thick,blue] (-2_1) --  (-2_1point5);
\draw[thick,red] (-2point5_1) --  (-2_1);
\draw[thick,red] (-2_1) --  (-1point5_1);

\node[circle, draw, fill=black,inner sep=1pt]  (2_-1point5) at (2,-1.5) {};
\node[circle, draw, fill=black,inner sep=1pt] (2point5_-1) at (2.5,-1) {};
\node[circle, draw, fill=black,inner sep=1pt] (1point5_-1) at (1.5,-1) {};
          \draw[thick,blue] (2_-1) --  (2_-1point5);
\draw[thick,red] (2_-1) --  (2point5_-1);
\draw[thick,red] (1point5_-1) --  (2_-1);

\node[circle, draw, fill=black,inner sep=1pt] (-2_-1point5) at (-2,-1.5) {};
\node[circle, draw, fill=black,inner sep=1pt] (-2point5_-1) at (-2.5,-1) {};
\node[circle, draw, fill=black,inner sep=1pt] (-1point5_-1) at (-1.5,-1) {};
          \draw[thick,blue] (-2_-1) --  (-2_-1point5);
\draw[thick,red] (-2_-1) --  (-2point5_-1);
\draw[thick,red] (-1point5_-1) --  (-2_-1);

\draw[thick,ForestGreen] (0_0) --  (2_1);
\draw[thick,ForestGreen] (1point5_1) -- (2_1point5);
\draw[thick,ForestGreen] (0_0) --  (-1_-2);
\draw[thick,ForestGreen] (-1point5_-2) --  (-1_-1point5);
\draw[thick,ForestGreen] (-0point5_-3) --  (0_-2);
\draw[thick,ForestGreen] (-2_0) --  (0_2);
\draw[thick,ForestGreen] (0_-2) --  (1_-1point5);
\draw[thick,ForestGreen] (-2point5_-1) -- (-2_0);
\draw[thick,ForestGreen] (0_2) -- (1_2point5);
\draw[thick,ForestGreen] (2_0) -- (3_0point5);
\draw[thick,ForestGreen]  (1point5_-1) -- (2_0);
\draw[thick,ForestGreen]  (-3point5_0) -- (-3_0point5);
\draw[thick,ForestGreen]  (-2point5_1) -- (-2_1point5);
\draw[thick,ForestGreen]  (-1point5_2) -- (-1_2point5);
\draw[thick,ForestGreen]  (-0point5_3) -- (0_3point5);
\draw[thick,ForestGreen]  (-3_0) -- (-2_1);
\draw[thick,ForestGreen]  (-1_2) -- (0_3);
\end{tikzpicture}}
\caption{$\text{Cay}(F(X),\Omega)$
\label{fig:f2shortcut}}

\end{center}
\end{subfigure}
\end{center}
      \caption{Geodetic Cayley graphs for $F(X)$ discussed in \cref{eg:shortcutinfreegroups}.
      \label{fig:f2}}
\end{figure}
\end{example}

\begin{example}\label{eg:cyclicorcompletegeodeticfinitegroups}
If $G$ is a finite group, we know of only two types of geodetic generating sets:  $G\setminus\{1_G\}$ which gives a complete graph as the Cayley graph; and for cyclic groups of odd order,  a single cyclic generator and its inverse gives an odd cycle as the Cayley graph. For example $\mathcal{C}_5=\langle a \mid a^5=1\rangle$ has geodetic generating sets $\{a^{\pm1}\}$ and $\{a,a^2,a^3,a^4\}$.
\end{example}

\begin{example}\label{eg:standardgeodeticgensetforplaingroup}
A \emph{plain group} is a free product of finitely many finite groups and a free group of finite rank. Let $G=G_1\ast\cdots\ast G_m \ast F(X)$, where the $G_i$ are finite groups and $X$ is a finite set. A 
geodetic generating set for $G$ is given by a disjoint union $\Sigma_1\sqcup\cdots\sqcup \Sigma_m\sqcup\Sigma_{F(X)}$, where $\Sigma_i$ are geodetic generating sets for the finite groups $G_i$ and $\Sigma_{F(X)}$ is a geodetic generating set for the free group $F(X)$. 
We will refer to a geodetic generating set for a free product where each factor has a disjoint geodetic generating set as \emph{standard}.
\end{example}

A natural question arises: are these the only possible geodetic Cayley graphs? Aside from the choice of $\Sigma_{F(X)}$ in the free factor, and choosing either a cycle or complete graph for the finite factors, before now, to the authors' knowledge,
no other way to construct geodetic Cayley graphs was known.

\section{Preliminaries}

In this section we list notation and basic results that will be used in this paper.

\subsection{Words and factors}

An \emph{alphabet} is a finite set.
If $\Sigma$ is an alphabet, $\Sigma^\ast$ denotes the set of all sequences $(x_1,\dots, x_n)$ of length $n$ with $x_i\in \Sigma$, which we write as words  $x_1\cdots x_n$.  
If $w\in \Sigma^\ast$ we write $|w|$ for the length of $w$. The word of length $0$ is denoted $\lambda$.  If $u,v\in \Sigma^\ast$ we say $v$ is a \emph{factor} of $u$ if there exist $\alpha,\beta\in\Sigma^\ast$ so that $u=\alpha v\beta$, and a \emph{proper factor} if $u\neq v$.

\subsection{Graphs and labelled graphs}
A \emph{(simple undirected) graph} $\Gamma$ comprises a set $V(\Gamma)$ of \emph{vertices} and a set $E(\Gamma)\subseteq \{Y \subseteq V(\Gamma) \mid |Y| = 2\}$ of \emph{(undirected) edges}.
Distinct vertices $u,v\in V(\Gamma)$ are \emph{adjacent} if $\{u,v\}\in E(\Gamma)$. 
A graph is \emph{locally-finite} if no vertex is adjacent to infinitely many other vertices.

\begin{definition}[Graph isomorphism]
    Let $\Gamma_1$ and $\Gamma_2$ be graphs. We say $\Gamma_1$ is isomorphic to $\Gamma_2$, denoted by $\Gamma_1\cong\Gamma_2$, if there exists a bijection $\pi: V(\Gamma_1)\to V(\Gamma_2)$ such that $\{v_1,v_2\}\in E(V_1)$ if and only if $\{\pi(v_1),\pi(v_2)\}\in E(\Gamma_2)$.
\end{definition}

For $a,b\in\mathbb Z$ with $a\leq b$, let $[a,b]$ denote the set $\{k\in\mathbb Z\mid a\leq k\leq b\}$.
A \emph{path} $p$ of \emph{length} $n\in \mathbb{N}$ in $\Gamma$ is a sequence of vertices $p=(v_0,v_1,\dots,v_n)$ such that for all 
$i \in [0,n-1]$ 
we have that $v_i$ is adjacent to $v_{i+1}$; we write $|p| = n$ and we say that $v_0$ is the \emph{initial vertex} and $v_n$ is the \emph{terminal vertex}. The sequence $p^{-1}=(v_n,v_{n-1},\dots,v_0)$ is called  
 the \emph{reverse path of $p$}.
A path with equal initial and terminal vertex is called a \emph{circuit}. A path that is not a circuit is \emph{embedded} if each vertex in the sequence of vertices is distinct. A circuit is \emph{embedded} if it has length at least three and the only non-distinct vertices are the initial and terminal vertices. In particular, if $u$ and $v$ are adjacent vertices in $\Gamma$, then the path $(u, v, u)$ is a circuit, but not an embedded circuit. A path is a \emph{geodesic} if its length is minimal among all paths with the same initial and terminal vertices. A graph $\Gamma$ is called \emph{geodetic} if, for each pair of vertices $u, v$, there exists a unique geodesic with initial vertex $u$ and terminal vertex $v$.

Let $\Gamma$ be a graph and $X$ a set.
A function $L\colon \{(u, v) \in V(\Gamma) \times V(\Gamma) \mid \{u, v\} \in E(\Gamma)\} \to X$ is called an \emph{$X$-labelling function} on $\Gamma$. 
An \emph{$X$-labelled graph} 
is a graph $\Gamma$ together with an $X$-labelling function $L$ on $\Gamma$. When the   $X$-labelling  function is understood, we refer to the $X$-labelled graph just as $\Gamma$.
For any path $p = (v_0, v_1, \dots, v_n)$ in an $X$-labelled graph $\Gamma$, the word $L(v_0, v_1) L(v_1, v_2) \dots L(v_{n-1}, v_n)$ is called the the \emph{label} on $p$; any path of length 0 is labelled by the empty word.

\begin{definition}[Labelled graph isomorphism]
Let $(\Gamma_i,L_i)$ be $X$-labelled graphs for $i\in\{1,2\}$. We say $(\Gamma_1,L_1)$ is isomorphic to $(\Gamma_2,L_2)$, denoted by $(\Gamma_1,L_1)\cong (\Gamma_2,L_2)$, if there is a graph isomorphism $\pi:V(\Gamma_1)\to V(\Gamma_2)$ that respects the labellings of the graph, that is,  $L_1(u,v))=L_2(\pi(u),\pi(v))$ for all $\{u,v\}\in E(\Gamma_1)$. 
\end{definition}

 \begin{definition}[\GamSlex{$\Gamma$}]\label{def:GammaSlex}
   Let $\Gamma$ be an $X$-labelled graph, and let 
 $\leq$ be a total order on 
 $X^*$ (typically, one would define a total order on $X$ which extends to a total order on $X^*$ in the obvious way).
 For $u,v\in X^\ast$ define $u\leq_{\text{slex}}v$ if $|u|\leq |v|$ and $u\leq v$.
 We say that a 
 path $\gamma$ in $\Gamma$ labelled by $u \in X^\ast$ 
 is 
 \emph{\GamSlex{$(\Gamma,\leq)$}}
 if $u\leq_{\text{slex}}v$ for all  $v \in X^\ast$ that label a path in $\Gamma$ with the same initial and terminal vertices as $u$. 
Since the choice of total order on $X$ is arbitrary, we will omit $\leq$ in the notation and simply say that  $\gamma$ is 
 \GamSlex{$\Gamma$}.
 \end{definition}

\subsection{Subdivisions}\label{subsec:subdiv} 

 Let $\Gamma_n$ be the graph obtained from $\Gamma$ by subdividing each edge into $2n+1$ sub-edges (so $\Gamma$ is a minor of $\Gamma_n$ obtained by contracting edges).  See for example \cref{fig:subdivisionofgraph}.
 Let 
 $f\colon \Gamma_n\to \Gamma$ be the homeomorphism corresponding to obtaining the minor,
 $\oldV=f^{-1}(V(\Gamma))$, and $\newV=V(\Gamma_n)\setminus \oldV$. Note that, by construction, embedded paths between vertices in $\oldV$ have length which is a multiple of $2n+1$, and therefore so do embedded circuits.

\begin{figure}[ht]
\begin{center}

 \begin{subfigure}{0.45\textwidth}
 \begin{center}
      \begin{tikzpicture}[rotate=18]
\def\r{1.7} 
\foreach \x in {-2,...,2}
\node[circle, draw, fill=black,inner sep=1.3pt] (\x) at (72*\x:\r) {};
\draw[thick] (-2) -- (-1);
\draw[thick] (-2) -- (2);
\draw[thick] (0) -- (-1);
\draw[thick] (0) -- (1);
\draw[thick] (1) -- (2);
\end{tikzpicture}
\caption{The $5$-cycle graph $\Gamma$.
\label{fig:Subdiv5-left}}
\end{center}
\end{subfigure}
      \begin{subfigure}{0.45\textwidth}
      \begin{center}
      \begin{tikzpicture}[rotate=18]
\def\r{1.7} 
\foreach \x in {-2,...,2}
\node[circle, draw, fill=black,inner sep=1.3pt] (\x) at (72*\x:\r) {};

\draw[thick] (-2) -- (-1)  node[pos=1/3,circle,color=gray,fill=gray,inner sep=1.3pt]{} node[pos=2/3,circle,color=gray,fill=gray,inner sep=1.3pt]{};
\draw[thick] (-2) -- (2) node[pos=1/3,circle,color=gray,fill=gray,inner sep=1.3pt]{} node[pos=2/3,circle,fill=gray,inner sep=1.3pt]{};;
\draw[thick] (0) -- (-1) node[pos=1/3,circle,fill=gray,color=gray,inner sep=1.3pt]{} node[pos=2/3,circle,color=gray,fill=gray,inner sep=1.3pt]{};;
\draw[thick] (0) -- (1) node[pos=1/3,circle,color=gray,fill=gray,inner sep=1.3pt]{} node[pos=2/3,circle,color=gray,fill=gray,inner sep=1.3pt]{};;
\draw[thick] (1) -- (2) node[pos=1/3,circle,color=gray,fill=gray,inner sep=1.3pt]{} node[pos=2/3,circle,color=gray,fill=gray,inner sep=1.3pt]{};

\end{tikzpicture}

\caption{The graph $\Gamma_1$ obtained by replacing each edge of $\Gamma$  by a path of length $2(1)+1=3$.
\label{fig:Subdiv5-right}}
\end{center}
\end{subfigure}
\end{center}
      \caption{The subdivision graph of a $5$-cycle graph with $n=1$.
      \label{fig:subdivisionofgraph}}
\end{figure}
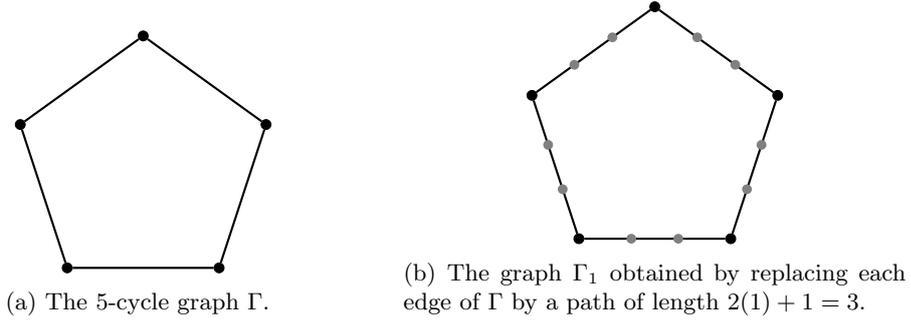

The next lemma  states
 that $\Gamma_n$ is geodetic if $\Gamma$ is geodetic. This can be extracted as a special case from the proof of \cite[Theorem 3.1]{parthasarathy1982some}, but we include a proof for completeness.

\begin{lemma}\label{lem:uniformsubdivision}
For each $n\in \mathbb{N}$, $\Gamma$ is geodetic if and only if $\Gamma_n$ is geodetic. 
\end{lemma}

\begin{proof}
Suppose $\Gamma_n$ is geodetic.  Let $p, q \in V(\Gamma)$.  There is a one-to-one  correspondence between embedded paths in $\Gamma$ with initial vertex $p$ and terminal vertex $q$, and embedded paths in $\Gamma_n$ with the corresponding initial and terminal vertices.  Under the correspondence, a path of length $k$ in $\Gamma$ corresponds to a path of length $k(2n+1)$ in $\Gamma_n$.  Since $\Gamma_n$ is geodetic, there is a unique shortest embedded path in $\Gamma_n$, and hence a unique shortest embedded path in $\Gamma$ with initial vertex $p$ and terminal vertex $q$.  Hence $\Gamma$ is geodetic.

Conversely, suppose $\Gamma$ is geodetic and for contradiction that $\Gamma_n$ is not geodetic. Then there exists distinct vertices $p,q\in V(\Gamma_n)$ and two distinct geodesics $\alpha,\beta$ of the same length with initial vertex $p$ and terminal vertex $q$. 
Recall that $\beta^{-1}$ denotes the reverse path of $\beta$. 
Without loss of generality we may assume that path $\alpha\beta^{-1}$ is an embedded circuit. 
 Hence,  $\alpha\beta^{-1}$ has even length in $\Gamma_n$ being a multiple of $2(2n+1)$. Thus, there exists a $k\in \mathbb{N}_+$ such that $d_{\Gamma_n}(p,q)=|\alpha|=|\beta|=k(2n+1)$.

If both $p,q\in \oldV$ then the length of $f(\alpha)$ and $f(\beta)$ in $\Gamma$ is $\frac{d_{\Gamma_n}(p,q)}{2n+1}=k$, and since $\Gamma$ is geodetic there exists a path $\gamma$ from $f(p)$ to $f(q)$ in $\Gamma$ with $|\gamma|<k$. But $f^{-1}(\gamma)$ has length $|\gamma|(2n+1)$ which is shorter than $k(2n+1)$, a contradiction.

Now without loss of generality let $p\in \newV$.
    Consider the sequence of vertices in $\oldV$ visited by $\alpha, \beta$ respectively. 
    Let $x_1,x_2$ (resp. $y_1,y_2$) be the first and last vertices in this sequence for $\alpha$ (resp. $\beta$). Note that $x_1\neq y_1$ since the first and last edges of $\alpha$ and $\beta$ are distinct and $p\in \newV$. 
    
    Since $d_{\Gamma_n}(x_1,y_1)=2n+1$ is odd, $p$ is closer to one of $x_1,y_1$  than to the other. Without loss of generality assume $d=d_{\Gamma_n}(x_1,p)<d_{\Gamma_n}(y_1,p)$, so $0<d\leq n$. Then \begin{align*}
        k(2n+1)&= |\alpha|\\
        &=d_{\Gamma_n}(p,x_1)+d_{\Gamma_n}(x_1,x_2)+d_{\Gamma_n}(x_2,q)  \\
        &=d+(2n+1)d_{\Gamma_n}(f(x_1),f(x_2))+d_{\Gamma_n}(x_2,q)
    \end{align*} so $d+d_{\Gamma_n}(x_2,q)$ is a multiple of $2n+1$. Since
    $0< d\leq  n$ and 
    $0< d, d_{\Gamma_n}(x_2,q)< 2n+1$, then $d+d_{\Gamma_n}=2n+1$, so 
$d_{\Gamma_n}(x_2,q)=2n+1-d$, and it follows that $d_{\Gamma_n}(y_2,q)=d_{\Gamma_n}(x_1,p)$.

Since we have two paths from $x_1$ to $y_2$ in $\Gamma_n$ each of length $k(2n+1)$, in $\Gamma$ we have two paths of length $k$ from $f(x_1)$ to $f(y_2)$, and since $\Gamma$ is geodetic there exists a path $\gamma$ from $f(x_1)$ to $f(y_2)$ that is shorter than $k$ (see \cref{fig:lemma-unif-subdiv}). This gives a path in $\Gamma_n$  (via $f^{-1}(\gamma)$) from $x_1$ to $y_2$ of length at most $(k-1)(2n+1)$, which gives a path from $p$ to $x_1$, to $y_2$ along $f^{-1}(\gamma)$, to $q$, which has 
length at most $2d+(k-1)(2n+1)$, which is less than $k(2n+1)$. This contradicts $\alpha,\beta$ being distinct geodesics in $\Gamma_n$.\end{proof}

\begin{figure}[ht]

      \begin{center}
      \begin{tikzpicture}[rotate=18]
\def\r{2.7} 
\foreach \x in {-3,...,2}
\node[circle, draw, fill=black,inner sep=1.3pt] (\x) at (60*\x:\r) {};

\draw[thick] (-3) -- (-2)  
node[pos=1/5,circle,color=gray,fill=gray,inner sep=1.3pt]{} 
node[pos=2/5,circle,color=gray,fill=gray,inner sep=1.3pt]{}
node[pos=3/5,circle,color=gray,fill=gray,inner sep=1.3pt]{} node[pos=4/5,circle,color=gray,fill=gray,inner sep=1.3pt]{};
\draw[thick] (-2) -- (-1)  
node[pos=3/5,above left]{\color{blue}$q$} node[pos=1/5,circle,color=gray,fill=gray,inner sep=1.3pt]{} node[pos=2/5,circle,color=gray,fill=gray,inner sep=1.3pt]{}
node[pos=3/5,circle,color=blue,fill=blue,inner sep=1.3pt]{} node[pos=4/5,circle,color=gray,fill=gray,inner sep=1.3pt]{};
\draw[thick] (-3) -- (2) 
node[pos=1/2,left]{$\alpha$} node[pos=1/5,circle,color=gray,fill=gray,inner sep=1.3pt]{} node[pos=2/5,circle,color=gray,fill=gray,inner sep=1.3pt]{}
node[pos=3/5,circle,color=gray,fill=gray,inner sep=1.3pt]{} node[pos=4/5,circle,color=gray,fill=gray,inner sep=1.3pt]{};
\draw[thick] (0) -- (-1) 
node[pos=1/5,circle,color=gray,fill=gray,inner sep=1.3pt]{} node[pos=2/5,circle,color=gray,fill=gray,inner sep=1.3pt]{}
node[pos=3/5,circle,color=gray,fill=gray,inner sep=1.3pt]{} node[pos=4/5,circle,color=gray,fill=gray,inner sep=1.3pt]{};
\draw[thick] (0) -- (1) 
node[pos=1/2, above right ]{$\beta$} node[pos=1/5,circle,color=gray,fill=gray,inner sep=1.3pt]{} node[pos=2/5,circle,color=gray,fill=gray,inner sep=1.3pt]{}
node[pos=3/5,circle,color=gray,fill=gray,inner sep=1.3pt]{} node[pos=4/5,circle,color=gray,fill=gray,inner sep=1.3pt]{};
\draw[thick] (1) -- (2) 
node[pos=3/5,above left]{\color{blue}$p$} node[pos=1/5,circle,color=gray,fill=gray,inner sep=1.3pt]{} node[pos=2/5,circle,color=gray,fill=gray,inner sep=1.3pt]{}
node[pos=3/5,circle,color=blue,fill=blue,inner sep=1.3pt]{} node[pos=4/5,circle,color=gray,fill=gray,inner sep=1.3pt]{};

 \draw[decorate, <-, color=red,
 decoration={snake, segment length=35mm, amplitude=2mm}
 ]           (-1)  -- (2)
 node[pos=1/2,above right  ]{$\gamma$} ;

\end{tikzpicture}

\end{center}
      \caption{Paths $\alpha,\beta,\gamma$ in the proof of \cref{lem:uniformsubdivision}.
      \label{fig:lemma-unif-subdiv}}
\end{figure}
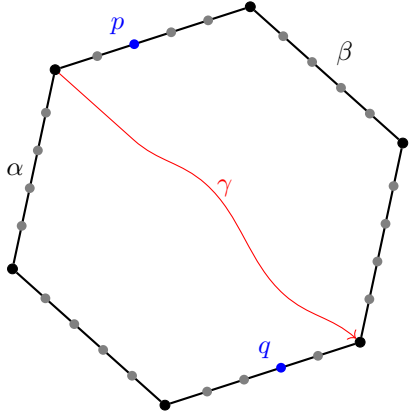

\subsection{Cayley graphs and geodetic groups}

Let $G$ be a group.  We write $1_G$ for the identity element in $G$, and we adopt the convention that the empty product is equal to $1_G$.  A subset $\Sigma \subset G$ is a \emph{finite inverse-closed generating set for $G$} if $\Sigma$ is finite, $\Sigma$ does not contain the identity element $1_G$, for every $g \in \Sigma$ we have that $g^{-1} \in \Sigma$,  and every element of $G$ is a product of elements in $\Sigma$.   Given such a set $\Sigma$, the \emph{Cayley graph of $G$ with respect to $\Sigma$} is the $\Sigma$-labelled graph $\Gamma = \text{Cay}(G, \Sigma)$ with $V(\Gamma) = G$, $E(\Gamma) = \{\{g, h\} \subset G \mid g^{-1} h \in \Sigma\}$, together with $\Sigma$-labelling function $L$ 
defined by $L((g,h))=g^{-1} h\in\Sigma$ for each $(g,h)\in V(\Gamma)\times V(\Gamma)$ with $\{g,h\}\in E(\Gamma)$.
It follows immediately that for any pair of adjacent vertices $g$ and $h$, $L(g, h)$ and $L(h, g)$ are inverses in $G$, and $L(h,g) = L(g,h) = x$ if and only if $x$ has order two.

As noted in the introduction, we call a group $G$ \emph{geodetic} if it admits a finite inverse-closed generating set $\Sigma$ such that 
$\text{Cay}(G, \Sigma)$ is geodetic; in which case we say that $\Sigma$ is a \emph{geodetic generating set for $G$}.

\subsection{Rewriting systems}

Following Otto and Book \cite{OttoBook}, 
a \emph{(string) rewriting system} is a pair $(\Sigma,T)$ where $\Sigma$ is an alphabet, and $T$ is a subset of $\Sigma^\ast\times\Sigma^\ast$ called the set of \emph{rewriting rules}.
The set $T$ determines a relation $\rightarrow$ on $\Sigma^\ast$ given by $u\ell v\rightarrow urv$ if $(\ell,r)\in T$ 
and  $u,v\in \Sigma^\ast$.
We write $\xrightarrow{\ast}$ for the reflexive and transitive closure of $\rightarrow$, and  $\overset{\ast}{\leftrightarrow}$ for the symmetric closure of $\xrightarrow{\ast}$. A word $u\in \Sigma^\ast$ is \emph{irreducible} if no factor is the left-hand side of any rewriting rule.

A rewriting system $(\Sigma,T)$ is called:
\bi
\item[(i)] \emph{finite} if $T$ is a finite set,

\item[(ii)]  \emph{length-reducing} if for each $(\ell, r)\in T$ we have $|\ell|>|r|$,

\item[(iii)] \emph{confluent} if for all $u,v,w\in \Sigma^{\ast}$,  $u\overset{\ast}{\leftarrow} w\overset{\ast}{\rightarrow} v$ implies that there exists $t\in \Sigma^{\ast}$ such that $u\overset{\ast}{\rightarrow}t\overset{\ast}{\leftarrow}v$,
\item[(iv)] \emph{Church-Rosser} if for all $u,v\in\Sigma^\ast$, $u\overset{\ast}{\leftrightarrow}v$ implies there exists $w\in \Sigma^\ast$ such that $u\overset{\ast}{\rightarrow}w\overset{\ast}{\leftarrow} v$.

\ei

The operation of concatenation is well defined on $\overset{\ast}{\leftrightarrow}$-classes, and so $\Sigma^\ast$ with concatenation forms a monoid. We denote this monoid by $M(\Sigma,T)$ and call it the monoid presented by $(\Sigma,T)$. 
If for every $x\in \Sigma^\ast$ there exists $y\in \Sigma^\ast$ such that $xy\overset{\ast}{\rightarrow}\lambda$ and $yx\overset{\ast}{\rightarrow}\lambda$, then we call $y$ the \emph{inverse} of $x$ in $M(\Sigma,T)$. We say $(\Sigma,T)$ is \emph{inverse-closed} if there is no $x\in \Sigma$ with $x\overset{\ast}{\leftrightarrow}\lambda$ and
for all $x\in \Sigma$ there exists an inverse $y\in \Sigma$.
Note that if $(\Sigma,T)$ is inverse-closed, then $M(\Sigma,T)$ is a group with inverse-closed generating set $\Sigma$.
Given a word $w=x_1x_2\cdots x_n\in \Sigma^\ast$ in an inverse-closed rewriting system $(\Sigma,T)$, let $w^{-1}=x_n^{-1}\cdots x_2^{-1}x_1^{-1}$ be the \emph{inverse word}.

We have the following connection between inverse-closed confluent length-reducing rewriting systems and geodetic groups.
\begin{lemma}\label{lem:icfclrrpresentsgroupwithgeodeticgeneratingset}
Let $(\Sigma, T)$ be an inverse-closed length-reducing rewriting system. Then $(\Sigma,T)$ is confluent if and only if $M(\Sigma,T)$ is a group with geodetic generating set $\Sigma$.
\end{lemma}

\begin{proof}
 Since $\Sigma$ is an inverse-closed, $M(\Sigma,T)$ is a group. Note that $(\Sigma,T)$ is confluent if and only if it is Church-Rosser (see \cite[Theorem 8.1 (d)]{DAM} or \cite[Lemma 1.17]{OttoBook}). That is, any word $v\in \Sigma^\ast$ rewrites to a unique form $w\in \Sigma^\ast$, which 
 is of minimal length since rules are length-reducing.
Now $v$ is the unique label of a geodesic path in the Cayley graph of the group $M(\Sigma,T)$ with finite inverse-closed generating set $\Sigma$. Hence, the result follows.
\end{proof}

\subsection{Group presentations}\label{subsec:grouppresentations}

A \emph{group presentation} is a pair $(X, R)$, usually written $\langle X \mid R \rangle$, comprising an alphabet $X$ and a set $R \subset (X \cup X^{-1})^\ast$ of words; the elements of $X$ are called \emph{(group) generators}, and the elements of $R$ are called \emph{relators}.
We denote by $\ggen{R}$ the smallest normal subgroup of the free group $F(X)$ that contains $R$. Note that $\ggen{R}=\gen{gr^{\pm 1} g^{-1}: g\in F(X), r\in R}$. 
The group $G$ presented by $\langle X \mid R \rangle$ is the quotient of $F(X)$ by $\ggen{R}$.  We note that, since $G$ is a quotient of $F(X)$, the words $x x^{-1}$ and $x^{-1}x$ spell the identity in $G$ for each generator $x$, and so these words need not appear in $R$.

\begin{lemma}\label{lem:rewritingsystemsandpresentations}
Let $(\Sigma,T)$ be 
an inverse-closed rewriting system and   $R=\{\ell r^{-1} \mid (\ell,r)\in T\}$.
Then 
the groups $M(\Sigma, T)$ and $G=\langle \Sigma \mid R\rangle$ isomorphic.      
\end{lemma}

\begin{proof}
Recall that each element of $M(\Sigma,T)$ is an 
 $\overset{\ast}{\leftrightarrow}$-equivalence class. Define $\phi\colon M(\Sigma,T)\to G$
 to be a map which takes
  any word $w\in\Sigma^*$ in a $\overset{\ast}{\leftrightarrow}$-class and, considering $w$ as an element of $F(\Sigma)$,  computes its  image under the quotient map to $G$. We now explain why the map is a well-defined group homomorphism. Suppose $w_1\overset{\ast}{\leftrightarrow} w_2$. Then there exists a finite sequence
\[w_1=w_{1,0}\leftrightarrow w_{1,1}\leftrightarrow w_{1,2}\leftrightarrow\cdots\leftrightarrow w_{1,k}=w_2,\] where each step is the application (or reverse) of a rule in $T$, that is, for each $i,j$ with $|i-j|=1$, $w_{1,i}=u\ell v$, $w_{1,j}=urv$ for some  $u,v\in\Sigma^\ast $ and $(\ell,r)\in T$.
Since $u\ell v, u\ell r^{-1}rv,urv$ each represent the same element of $G$,  each 
 $w_{1,i}$ represents the same element of $G$ as $w_1$ and $w_2$.

It is clear that $\phi$ is 
  surjective, since $G$ is generated by $\Sigma$.  For injectivity, let $w\in \Sigma^\ast$ represent the identity of $G$. Hence, $w=\prod_{i=1}^{n} u_i \gamma_i u_i^{-1}$ for some $n\in \mathbb{N}$, $u_i\in \Sigma^\ast$ and $\gamma_i\in R$ (see \cite[Prop. 4.1.2]{BridsonGeomofWordProb}). We wish to show that $w$ represents the identity in $M(\Sigma,T)$, that is, $w\overset{\ast}{\leftrightarrow}\lambda$. 
Using the rewriting we 
find that\begin{align*}
 w&=u_1 \gamma_1 u_1^{-1} u_2 \gamma_2 u_2^{-1} \cdots u_n \gamma_n u_n^{-1} \\
    &=u_1 \ell_1r_1^{-1}  u_1^{-1}  u_2 \ell_2r_2^{-1} u_2^{-1}\cdots u_n\ell_n r_n^{-1} u_n^{-1} \\
    &\overset{\ast}{\to}u_1 r_1r_1^{-1}  u_1^{-1}  u_2 r_2r_2^{-1} u_2^{-1}\cdots u_nr_n r_n^{-1} u_n^{-1} \\
    &\overset{\ast}{\to} u_1   u_1^{-1}  u_2 u_2^{-1}\cdots u_n u_n^{-1} \overset{\ast}{\to}\lambda.
\end{align*}
Hence, $w$ represents the identity in $M(\Sigma,T)$ and so $\phi$ is injective.
\end{proof}

It is a standard fact in combinatorial group theory that 
the following transformations of group presentations preserve the groups they present up to isomorphism (see for example \cite[Pg. 89]{LyndonSchupp}). 

\begin{definition}\label{def:tietzetransformations}
    Given a presentation $\langle X\mid R\rangle$, the following transformations of presentations are called \emph{Tietze transformations}:
\begin{description}
	\item[(T1)]
	Replace $\langle X\mid R\rangle$ by $\langle X\mid R\cup S \rangle$, where $S\subset \ggen{R}$.
	\item[(T1$^{-1}$)]
Replace $\langle X\mid R\rangle$ by $\langle X \mid R\setminus S\rangle$, where $S\subset \ggen{R \setminus S} $.
	\item[(T2)] Replace $\langle X\mid R\rangle$ by $\langle X\cup Y\mid R\cup S\rangle$, where $Y\cap F(X)=\emptyset$ is a set of new symbols, 
    and $S=\{y^{-1}u_y\colon y\in Y, u_y\in (X\cup X^{-1})^\ast \}$.
	\item[(T2$^{-1}$)]
	Replace $\langle X\mid R\rangle$ by $\langle X\setminus Y\mid R\setminus S\rangle$, where $Y\subseteq X$ such that $S=\{y u_y\colon y\in Y, u_y\in ((X\setminus Y)\cup (X\setminus Y)^{-1})^\ast\}\subseteq R$.
    \end{description}
\end{definition}

\begin{example}\label{eg:C4Tietze}
The cyclic group of order four has presentation $\langle x \mid x^4\rangle$. Applying \textbf{(T2)} with $Y=\{y,z\}$ and $S=\{y^{-1}x^2, z^{-1}x^3\}$
yields a presentation
\begin{equation}
    \label{eq:presC4Tietze}
\langle x,y,z\mid x^4, y^{-1}x^2,z^{-1}x^3\rangle\end{equation}
for the same group (up to isomorphism).
Similarly, in \cref{eg:shortcutinfreegroups}, applying \textbf{(T2)} with $Y=\{z\}$ and $S=\{z^{-1}xy\}$ to  $\langle x,y\mid -\rangle$ gives the presentation whose corresponding Cayley graph is shown in \cref{fig:f2shortcut}. 
\end{example}
The following well-known lemma (see \cite[Pg. 174]{LyndonSchupp}) records a way in which a free product decomposition may be seen from a group presentation.

\begin{lemma}\label{lem:FreeProducts}
If $\langle X_1 \mid R_1\rangle$ and $\langle X_2 \mid R_2\rangle$ are group presentations such that $X_1 \cap X_2 = \emptyset$, then 
\[\langle X_1 \cup X_2 \mid R_1 \cup R_2 \rangle \cong \langle X_1 \mid R_1\rangle \ast \langle X_2 \mid R_2\rangle.\]
\end{lemma}

\section{The construction}\label{sec:TheConstruction}

In this section we describe the key construction for the paper. 
Throughout this section, let 
 $G$  be a group with a finite inverse-closed generating set $\Sigma$, $(\Gamma,L) = \text{Cay}(G, \Sigma)$, and let $n$ be a natural number.   

First, let us introduce the following notation for the elements of $\Sigma$.  Let $\Sigma_1$ be the subset of $\Sigma$ comprising the elements of order two.  If $\Sigma_1 = \Sigma$, then define $\Sigma_2 = \Sigma_3 = \emptyset$; otherwise, let $\{\Sigma_2, \Sigma_3\}$ be a partition of $\Sigma \setminus \Sigma_1$ such that for each $x \in \Sigma_2$, we have $x^{-1} \in \Sigma_3$.  Let $m_1 = |\Sigma_1|$, $m_2 = |\Sigma_2|$, and introduce symbols for the elements of $\Sigma$ such that $\Sigma_1=\{a_1,\dots,a_{m_1}\}$, $\Sigma_2=\{b_1,\dots,b_{m_2}\}$, $\Sigma_3=\{c_1,\dots,c_{m_2}\}$, and $c_i$ is the inverse of $b_i$ in $G$ for all $i \in [1,m_2]$.

Next, we introduce a larger set of symbols which will become the labelling set for a subdivision  of $\Gamma$. 
Let $\Sigma_1(n) = \{a_{i, j} \mid i\in[1,m_1], j\in[1, 2n+1]\}$, $\Sigma_2(n) = \{b_{i, j} \mid i\in[1,m_2], j\in[1, 2n+1]\}$, $\Sigma_3(n) = \{c_{i, j} \mid i\in[ 1, m_2], j\in[1, 2n+1]\}$
and let $\Sigma(n) = \Sigma_1(n) \cup \Sigma_2(n) \cup \Sigma_3(n)$.

Let $\Gamma_n$ denote the graph obtained from $\Gamma$ by subdividing each edge into $2n+1$ sub-edges (as discussed in \cref{subsec:subdiv}). Let $L(n)$ be the $\Sigma(n)$-labelling of $\Gamma_n$ constructed such that the following holds:  for every edge $\{g, h\} \in E(\Gamma)$ with label $L((g,h))=x_i \in \Sigma_1 \cup \Sigma_2$, the path in $\Gamma_n$ from $g$ to $h$ of length $2n+1$ is labelled $x_{i, 1}, x_{i, 2}, \dots, x_{i, 2n+1}$; for every edge $\{g, h\} \in E(\Gamma)$ with label $L((g,h))=c_i \in \Sigma_3$, the path in $\Gamma_n$ from $g$ to $h$ of length $2n+1$ is labelled $c_{i, 2n+1}, c_{i, 2n}, \dots, c_{i,1}$.
Assume a total order on $\Sigma(n)^*$
has been fixed so that  \GamSlex{$\Gamma_n$} is well defined  as in 
 \cref{def:GammaSlex}.

We define a set of rewriting rules $T(n) \subset \Sigma(n)^\ast \times \Sigma(n)^\ast$ as follows: 
$(\ell, r)\in T(n)$ 
if and only if one of the following is true: 
\begin{enumerate}
\item [$(\mathcal{R}_1)$] $\ell$ is the label on a path $(u, v, u)$ for some pair of adjacent vertices $u$ and $v$, and $r = \lambda$;
\item [$(\mathcal{R}_2)$] $\ell$ is the label on a path $\alpha$, $r$ is the label on a path $\beta$, $\alpha\beta^{-1}$ is 
 an embedded circuit  of odd length, $|\alpha| = |\beta| + 1$, and $\beta$ is \GamSlex{$\Gamma_n$}  
\item [$(\mathcal{R}_3)$] $\ell$ is the label on a path $\alpha$, $r$ is the label on a path $\beta$, $\alpha\beta^{-1}$ is 
 an embedded circuit of length at least $4$, $|\alpha| = |\beta|$, and $\beta$ is \GamSlex{$\Gamma_n$}.
\end{enumerate}
Finally, we write $\nabla_n(G,\Sigma)$ for the rewriting system $(\Sigma(n),T(n))$.

\begin{remark}\label{rem:ConstructionRemark}
Note that rules of type $(\mathcal{R}_1)$ and $(\mathcal{R}_2)$ never label circuits made up of two distinct geodesics with the same initial and terminal vertices. Hence, $\Gamma_n$ is geodetic if and only if there are no rules of type $(\mathcal{R}_3)$ in $\nabla_n(G,\Sigma)$.  
\end{remark}

\begin{example}\label{eg:nabla1ofC4}
Let $G$ denote the cyclic group of order 4, and 
recall the presentation \cref{eq:presC4Tietze} for $G$ given in  \cref{eg:C4Tietze}. 
Let $\Sigma = \{x, y, z\}$ and $n=1$. Since $\Gamma = \text{Cay}(G, \Sigma)$ is the complete graph $K_4$, $G$ is geodetic. The graph $\Gamma$ with $\Sigma$-labelling $L$ is shown in \cref{fig:biggerk4}.

\begin{figure}[ht] 
\begin{center} \scalebox{0.75}{
           \begin{tikzpicture}
\node[below] at (0,0) {$1_G$};
\node[above left] at (0,0.2)  {\textcolor{blue}{$x$}};
\node[left] at (-0.2,0)  {\textcolor{red}{$z$}};
\node[circle, draw, fill=black,inner sep=1.3pt] (A) at (0,0) {};
\node[left] at (-5,3) {$x$};
\node[below] at (-5,3) {\textcolor{blue}{$x$}};
\node[right] at (-4.7,2.9)  {\textcolor{red}{$z$}};
\node[circle, draw, fill=black,inner sep=1.3pt] (B) at (-5,3) {};
\node[right] at (3.5,5) {$y$};
\node[left] at (3.3,5.1)  {\textcolor{blue}{$x$}};
\node[above] at (3.5,5.1)  {\textcolor{red}{$z$}};
\node[circle, draw, fill=black,inner sep=1.3pt] (C) at (3.5,5) {};
\node[above] at (0,8) {$z$};
\node[right] at (0.2,8)  {\textcolor{blue}{$x$}};
\node[below left] at (0,7.8)  {\textcolor{red}{$z$}};
\node[circle, draw, fill=black,inner sep=1.3pt] (D) at (0,8) {};

 \path[thick,{Stealth[color=red]}-{Stealth[color=blue]}] (A) edge  (B);

  \path[thick,{Stealth[color=ForestGreen]}-{Stealth[color=ForestGreen]}] (D) edge  node[above left] {\textcolor{ForestGreen}{$y$}}  (B);

    \path[thick,{Stealth[color=ForestGreen]}-{Stealth[color=ForestGreen]}] (A) edge node[below right]  {\textcolor{ForestGreen}{$y$}}  (C);

 \path[thick,{Stealth[color=red]}-{Stealth[color=blue]}] (C) edge (D);

 \path[thick,{Stealth[color=blue]}-{Stealth[color=red]}] (A) edge  (D);

 \path[thick,dashed,{Stealth[color=red]}-{Stealth[color=blue]}] (B) edge (C);

\end{tikzpicture}}
\end{center}
\caption{$\Gamma=\text{Cay}(G,\Sigma)$ with  $\Sigma$-labelling $L$.}
   
   \label{fig:biggerk4}
\end{figure}
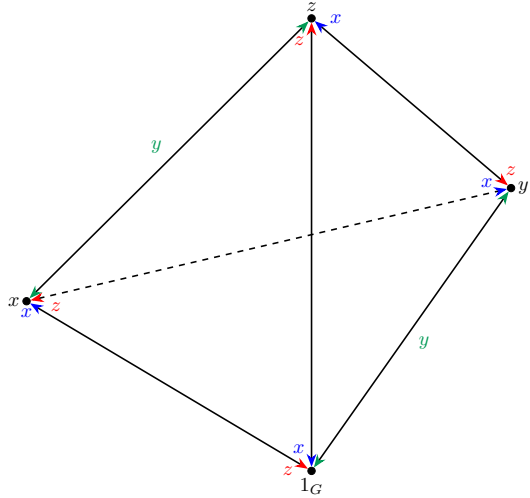

We have $\Sigma_1 = \{y\}, \Sigma_2 = \{x\}$ and $\Sigma_3 = \{z\}$. Following the notation in the construction we let $a_1 = y$, $b_1 = x$ and $c_1 = z$ and so we have $\Sigma_1(1) = \{a_{1, 1}, a_{1, 2}, a_{1, 3}\}$, $\Sigma_2(1) = \{b_{1, 1}, b_{1, 2}, b_{1, 3}\}$, $\Sigma_3(1) = \{c_{1, 1}, c_{1, 2}, c_{1, 3}\}$ and \[\Sigma(1) = \{a_{1, 1}, a_{1, 2}, a_{1, 3}, b_{1, 1}, b_{1, 2}, b_{1, 3}, c_{1, 1}, c_{1, 2}, c_{1, 3}\}.\]  The graph $\Gamma(1)$ with $\Sigma(1)$-labelling $L(1)$ is shown in \cref{fig:biggerk4subdiv}. 

\begin{figure}[ht] 
   \begin{center}
 \scalebox{1.25}{\begin{tikzpicture}
\node[below] at (0,0) {$1_G$};
\node[above left] at (0,0.2)  {\tiny{\textcolor{blue}{$b_{1,3}$}}};
\node[left] at (-0.2,0)  {\tiny{\textcolor{red}{$c_{1,1}$}}};
\node[right] at (0.2,0)  {\tiny{\textcolor{Peach}{$a_{1,3}$}}};
\node[circle, draw, fill=black,inner sep=1.3pt] (A) at (0,0) {};
\node[left] at (-5,3) {$x$};
\node[below] at (-4.9,2.9) {\tiny{\textcolor{blue}{$b_{1,3}$}}};
\node[right] at (-4.7,2.9)  {\tiny{\textcolor{red}{$c_{1,{1}}$}}};
\node[above right] at (-5.3,3.2)  {\tiny{\textcolor{Peach}{$a_{1,3}$}}};
\node[circle, draw, fill=black,inner sep=1.3pt] (B) at (-5,3) {};
\node[right] at (3.5,5) {$y$};
\node[left] at (3.3,5.1)  {\tiny{\textcolor{blue}{$b_{1,3}$}}};
\node[above] at (3.5,5.1)  {\tiny{\textcolor{red}{$c_{1,{1}}$}}};
\node[below right] at (3.3,4.8)  {\tiny{\textcolor{Peach}{$a_{1,3}$}}};

\node[circle, draw, fill=black,inner sep=1.3pt] (C) at (3.5,5) {};
\node[above] at (0,8) {$z$};
\node[right] at (0.2,8)  {\tiny{\textcolor{blue}{$b_{1,3}$}}};
\node[below] at (0.3,7.7)  {\tiny{\textcolor{red}{$c_{1,{1}}$}}};
\node[left] at (-0.2,7.9)  {\tiny{\textcolor{Peach}{$a_{1,3}$}}};
\node[circle, draw, fill=black,inner sep=1.3pt] (D) at (0,8) {};

\draw[thick,color=white] (A) -- (B)
node(AB1)[pos=1/3,draw,color=gray,circle,fill=gray,inner sep=1.3pt]{} 
node(AB2)[pos=2/3,draw,color=gray,circle,fill=gray,inner sep=1.3pt]{} 
;
\node at (-1.5,0.6)  {\tiny{\textcolor{blue}{$b_{1,1}$}}};
\node at (-3.2,1.6)  {\tiny{\textcolor{blue}{$b_{1,2}$}}};
\node at (-2,0.8)  {\tiny{\textcolor{red}{$c_{1,2}$}}};
\node at (-3.8,1.9)  {\tiny{\textcolor{red}{$c_{1,3}$}}};

\draw[thick,color=white] (A) -- (C)
node(AC1)[pos=1/3,draw,color=gray,circle,fill=gray,inner sep=1.3pt]{} 
node(AC2)[pos=2/3,draw,color=gray,circle,fill=gray,inner sep=1.3pt]{} 
;
\node[below right] at (1.1,1.6)  {\tiny{\textcolor{Sepia}{$a_{1,1}$}}};
\node[above right] at (2.4,3.1)  {\tiny{\textcolor{Sepia}{$a_{1,1}$}}};

\draw[thick,color=white] (A) -- (D)
node(AD1)[pos=1/3,draw,color=gray,circle,fill=gray,inner sep=1.3pt]{} 
node(AD2)[pos=2/3,draw,color=gray,circle,fill=gray,inner sep=1.3pt]{} 
;
\node at (-0.3,3)  {\tiny{\textcolor{blue}{$b_{1,2}$}}};
\node at (-0.3,5.5)  {\tiny{\textcolor{blue}{$b_{1,1}$}}};
\node at (-0.3,2.5)  {\tiny{\textcolor{red}{$c_{1,{3}}$}}};
\node at (-0.3,5.1)  {\tiny{\textcolor{red}{$c_{1,2}$}}};

\draw[thick,color=white] (B) -- (C)
node(BC1)[pos=1/3,draw,color=gray,circle,fill=gray,inner sep=1.3pt]{} 
node(BC2)[pos=2/3,draw,color=gray,circle,fill=gray,inner sep=1.3pt]{} 
;
\node at (-2.3,3.3)  {\tiny{\textcolor{blue}{$b_{1,1}$}}};
\node at (0.5,4)  {\tiny{\textcolor{blue}{$b_{1,2}$}}};
\node at (-1.7,3.4)  {\tiny{\textcolor{red}{$c_{1,2}$}}};
\node at (1,4.1)  {\tiny{\textcolor{red}{$c_{1,3}$}}};

\draw[thick,color=white] (B) -- (D)
node(BD1)[pos=1/3,draw,color=gray,circle,fill=gray,inner sep=1.3pt]{} 
node(BD2)[pos=2/3,draw,color=gray,circle,fill=gray,inner sep=1.3pt]{} 
;
\node at (-3.8,4.7)  {\tiny{\textcolor{Sepia}{$a_{1,1}$}}};
\node at (-1.8,6.6)  {\tiny{\textcolor{Sepia}{$a_{1,1}$}}};

\draw[thick,color=white] (C) -- (D)
node(CD1)[pos=1/3,draw,color=gray,circle,fill=gray,inner sep=1.3pt]{} 
node(CD2)[pos=2/3,draw,color=gray,circle,fill=gray,inner sep=1.3pt]{} 
;

\node at (1.7,7)  {\tiny{\textcolor{blue}{$b_{1,2}$}}};
\node at (2.8,6)  {\tiny{\textcolor{blue}{$b_{1,1}$}}};
\node at (1.2,7.4)  {\tiny{\textcolor{red}{$c_{1,{3}}$}}};
\node at (2.3,6.4)  {\tiny{\textcolor{red}{$c_{1,2}$}}};

\path[thick,{Stealth[color=red]}-{Stealth[color=blue]}] (A) edge  (AB1);
\path[thick,{Stealth[color=red]}-{Stealth[color=blue]}] (AB1) edge (AB2);
\path[thick,{Stealth[color=red]}-{Stealth[color=blue]}] (AB2) edge  (B);

\path[thick,{Stealth[color=red]}-{Stealth[color=blue]}] (C) edge (CD1);
\path[thick,{Stealth[color=red]}-{Stealth[color=blue]}] (CD1) edge (CD2);
\path[thick,{Stealth[color=red]}-{Stealth[color=blue]}] (CD2) edge  (D);

\path[thick,dashed,{Stealth[color=red]}-{Stealth[color=blue]}] (B) edge (BC1);
\path[thick,dashed,{Stealth[color=red]}-{Stealth[color=blue]}] (BC1) edge (BC2);
\path[thick,dashed,{Stealth[color=red]}-{Stealth[color=blue]}] (BC2) edge  (C);

\path[thick,{Stealth[color=blue]}-{Stealth[color=red]}] (A) edge  (AD1);
\path[thick,{Stealth[color=blue]}-{Stealth[color=red]}] (AD1) edge  (AD2);
\path[thick,{Stealth[color=blue]}-{Stealth[color=red]}] (AD2) edge   (D);

\path[thick,{Stealth[color=Peach]}-{Stealth[color=Sepia]}] (A) edge   (AC1);
\path[thick,{Stealth[color=ForestGreen]}-{Stealth[color=ForestGreen]}] (AC1) edge node[below right] {\textcolor{ForestGreen}{\tiny{$a_{1,2}$}}}  (AC2);
\path[thick,{Stealth[color=Sepia]}-{Stealth[color=Peach]}] (AC2) edge  (C);

\path[thick,{Stealth[color=Peach]}-{Stealth[color=Sepia]}] (B) edge (BD1);
\path[thick,{Stealth[color=ForestGreen]}-{Stealth[color=ForestGreen]}] (BD1) edge node[above left] {\textcolor{ForestGreen}{\tiny{$a_{1,2}$}}}  (BD2);
\path[thick,{Stealth[color=Sepia]}-{Stealth[color=Peach]}] (BD2) edge   (D);

\end{tikzpicture}}
\end{center}
\caption{$\Gamma(1)$ with $\Sigma(1)$-labelling $L(1)$.}
   \label{fig:biggerk4subdiv}
\end{figure}

As a demonstration, here are 
 some of the rewriting rules in $T(1)$:
\begin{multline*}
T(1) = \{\underset{\text{
rules of type $\mathcal{R}_1$, each coming from traversing an edge back and forth}}{\underbrace{(a_{1, 1} a_{1, 3}, \lambda), \;\;\; (a_{1, 3} a_{1, 1}, \lambda),  \;\;\; (a_{1, 2} a_{1, 2}, \lambda), \;\;\; (b_{1, 1} c_{1, 1}, \lambda)  \;\;\;(c_{1, 1} b_{1, 1}, \lambda), \;\;\; \dots,}}\\ \underset{\text{
rules of type $\mathcal{R}_2$ that describe shortcuts in odd-length embedded circuits}}{ \underbrace{(a_{1, 1} a_{1, 2} a_{1, 3} b_{1, 1} b_{1, 2}, \ {c_{1,3}} c_{1, 2} c_{1,1} c_{1, 3}),\;\;\; 
(a_{1, 2} a_{1, 3} b_{1, 1} b_{1, 2} b_{1, 3}, \ a_{1, 3} c_{1, 3} c_{1, 2} c_{1,1}), \;\;\; 
\dots }}\}
\end{multline*}
Recall  (\cref{rem:ConstructionRemark}) since $\Gamma$ is geodetic there are no rules of type $(\mathcal{R}_3)$ in $\nabla_1(G,\Sigma)$. 
\end{example}

\begin{lemma}\label{lem:geodesictominimal}
    Let $w$ be the label of a geodesic path in $\Gamma_n$. Let $w'$ be the label of a 
    \GamSlex{$\Gamma_n$} 
    path that shares the same initial and terminal vertices as $w$. Then $w\overset{\ast}{\to} w'$.
\end{lemma}

\begin{proof}
 If $w$ is a label of a \GamSlex{$\Gamma_n$} path, we are done.
If not, scan $w$ from left to right until the prefix labels a path that
is not \GamSlex{$\Gamma_n$}. Take this prefix and scan right to left until the suffix labels a path  $\gamma$ that is not \GamSlex{$\Gamma_n$}. Let $p$ be the label of $\gamma$. There exists a path $\gamma_0$ with the same endpoints as $\gamma$,  that is \GamSlex{$\Gamma_n$}, such that $\gamma\gamma_0^{-1}$ labels an embedded circuit of even length at least $4$. Let $p_0$ be the label of $\gamma_0$. Hence $(p,p_0)\in T(n)$ is of type ($\mathcal{R}_3$). Let $w_1$ be the word $w$ with the $p$ factor replaced by $p_0$. If $w_1$ is a label of a \GamSlex{$\Gamma_n$} path we are done, otherwise repeat. The procedure will terminate as $\leq$ is a total order on $\Sigma(n)^\ast$, at each step there is a reduction with respect to $\leq$, and the length of words is bounded below by $0$.
\end{proof}

\begin{lemma}\label{lem:pathtoslex}
    Let $w$ be the label of a path in $\Gamma_n$. Let $w''$ be the label of a \GamSlex{$\Gamma_n$} path with the same initial and terminal vertices as $w$. Then $w\overset{\ast}{\to} w''$.
\end{lemma}

\begin{proof}
      If $w$ is a label of a \GamSlex{$\Gamma_n$} path, we are done. If $w$ is a label of a geodesic path, apply \cref{lem:geodesictominimal} and we are done. Otherwise, scan $w$ from left to right, until a prefix is not a geodesic. Now scan the prefix from right to left until a suffix is not a geodesic. Let $p$ be this suffix. Either $p$ is a label on a path $(u,v,u)$ for some pair of adjacent vertices $u$ and $v$, and so we replace the path labelled by $p$ by the empty path using the rule $(p,\lambda)$, or there exists a path labelled by $p_0$ that is \GamSlex{$\Gamma_n$} and $pp_0^{-1}$ labels  an embedded circuit in $\Gamma_n$, so $(p,p_0)\in T(n)$, and replace the path labelled by $p$ by the path labelled by $p_0$. Let $w_1$ be the label of the word obtained. If $w_1$ is a label of a \GamSlex{$\Gamma_n$} path, we are done. If $w_1$ is a label of a geodesic path, apply \cref{lem:geodesictominimal} and we are done. Otherwise we repeat the process. The procedure will terminate as $\leq$ is a total order on $\Sigma(n)^\ast$, at each step there is a reduction with respect to $\leq$, and word length is bounded below by $0$.
\end{proof}

\begin{lemma}\label{lem:nablaisICRS}
   Let $G$ be a group with finite inverse-closed generating set $\Sigma$. For each $n\in \mathbb{N}$, $\nabla_n(G,\Sigma)$ is an inverse-closed confluent rewriting system. 
\end{lemma}

\begin{proof}
The inverse of each $a_{i,j}\in \Sigma_1(n)$ is $a_{i,2n+2-j}\in \Sigma_1(n)$. The inverse of each $b_{i,j}\in \Sigma_2(n)$ is $c_{i,2n+2-j}\in \Sigma_3(n)$. Hence, each element of $\Sigma(n)$ has an inverse, so the rewriting system is inverse-closed.

 Suppose $u,v,w\in \Sigma^{\ast}$ with $u\overset{\ast}{\leftarrow} w\overset{\ast}{\rightarrow} v$ and $u\neq v$. By the construction, $u$ and $v$ are labels of paths in $\Gamma_n$ with the same initial and terminal vertices of $w$. By \cref{lem:pathtoslex}, there are \GamSlex{$\Gamma_n$}  geodesics with labels $u_1$ and $v_1$ and the same initial and terminal vertices of $w$ so that $u\overset{\ast}{\to}u_1$ and $v\overset{\ast}{\to} v_1$.  Since $\leq$ is a total order, we have $u_1=v_1$. Hence, $\nabla_n(G,\Sigma)$ is confluent.
 \end{proof}

Now that we know that $\nabla_n(G,\Sigma)$ is an inverse-closed rewriting system, we can conclude that the  monoid presented by $\nabla_n(G,\Sigma)$  is a group with finite inverse-closed generating set $\Sigma(n)$. Let \begin{equation}\label{eq:groupconstructed}
    \mathcal{G}_n(G,\Sigma)=M(\nabla_n(G,\Sigma))
\end{equation}
\begin{equation}\label{eq:groupgenconstructed}
    \mathcal{G}^{\text{gen}}_n(G,\Sigma)=(\mathcal{G}_n(G,\Sigma),\Sigma(n))
\end{equation}
In \cref{sec:identifythegroup}, we will identify the isomorphism class of $\mathcal{G}_n(G,\Sigma)$, while in \cref{sec:propertiesoftheconstruction}, we will prove some properties of Cayley graphs associated to $\mathcal{G}^{\text{gen}}_n(G,\Sigma)$. A notational advantage of  $\mathcal{G}_n^{\text{gen}}$ is that the input and output are both a group and its generating set, so can be applied repeatedly.

\begin{proposition}\label{prop:nablaisICCLRRS}
    Let $G$ be a group with finite inverse-closed generating set $\Sigma$. Then for each $n\in \mathbb{N}$, $\text{Cay}(G,\Sigma)$ is geodetic if and only if $\nabla_n(G,\Sigma)$ is length-reducing.
\end{proposition}

\begin{proof}
As noted in \cref{rem:ConstructionRemark}, $\Gamma_n$ is geodetic if and only if there are no rules of type $(\mathcal{R}_3)$.

Suppose $\Gamma=\text{Cay}(G,\Sigma)$ is geodetic. By \cref{lem:uniformsubdivision}, $\Gamma_n$ is geodetic. Since $\Gamma_n$ is geodetic, there are no rules of type $(\mathcal{R}_3)$ and we conclude $\nabla_n(G,\Sigma)$ is length-reducing.

Suppose $\Gamma=\text{Cay}(G,\Sigma)$ is not geodetic. By \cref{lem:uniformsubdivision}, $\Gamma_n$ is not geodetic, so $\mathcal{R}_3$ is non-empty. Hence,  $\nabla_n(G,\Sigma)$ is not length-reducing.
 \end{proof}

We have now proved the following. 
\setcounter{theoremx}{1}
\begin{theoremx}\label{thm:constructiongeodeticiffgeodetic}
  Let $G$ be a group with finite inverse-closed generating set $\Sigma$. For each $n\in \mathbb{N}$, $\text{Cay}(G,\Sigma)$ is geodetic if and only if $\text{Cay}(\mathcal{G}_n^{\text{gen}}(G,\Sigma))$ is geodetic.
\end{theoremx}

\begin{proof}
    Follows from \cref{lem:icfclrrpresentsgroupwithgeodeticgeneratingset}, \cref{lem:nablaisICRS} and \cref{prop:nablaisICCLRRS}.
\end{proof}

\section{Identifying the group constructed}\label{sec:identifythegroup}

We will use Tietze transformations to transform  $\langle \Sigma \mid \{\ell^{-1}r \mid (\ell,r)\in T\}\rangle$ into a presentation from which the the isomorphism class of $\mathcal{G}_n(G,\Sigma)$ is easily understood (that is, as in \cref{lem:FreeProducts}).

\begin{proposition}\label{prop:whatgroupisnablapresenting}
        Let $n\in \mathbb{N}$ and let $G$ be a group with finite inverse-closed generating set $\Sigma$. Then $\mathcal{G}_n(G,\Sigma)\cong G\ast F_{n|\Sigma|}$. 
\end{proposition}

\begin{proof}
By \cref{lem:rewritingsystemsandpresentations}, the group $\nabla_n(G,\Sigma)$ associated to the rewriting system $(\Sigma(n),T(n))$ is the group presented by $\langle \Sigma(n) \mid R_0\rangle$ where $R_0=\{\ell r^{-1} \mid (\ell,r)\in T(n)\}$.   We adopt the notation of \cref{sec:TheConstruction}.  Note that rewriting rules of type $(\mathcal{R}_1)$ express inverse relations between generators, while all other rewriting rules label embedded circuits in $\Gamma(n)$.

As a first step, we perform a sequence of Tietze transformations that remove generators which are inverses of other generators. Let $X_0=\Sigma(n)$ and $X_1 = \{a_{i, j} \mid i\in [1, m_1],  j \in[1,n+1]\} \cup \Sigma_2(n)$;  we note that $X_1$ is a minimal set with the property that the inverse closure of $X_1$ is $X_0$. 
Let $R_0'=\{a_{i,2n+1-j}a_{i,j}\mid i\in[1,m_1],j\in[1,n]\}\cup \{c_{i,j}b_{i,j}\mid i\in[1,m_2],j\in [1,2n+1]\}$. Note that each $yx\in R_0'$ corresponds to some path $(u,v,u)$ of $\Gamma_n$.
Let $R_1\subseteq X_0^\ast$ be the set of words obtained from $R_0$ by applying a letter homomorphism which replaces $a_{i,2n+1-j}$ by $a_{i,j}^{-1}$ for $i\in[1,m_1],j\in[1,n]$, and $c_{i,j}$ by $b_{i,j}^{-1}$ for $i\in[1,m_2],j\in [1,2n+1]$. It is clear that $R_0'\subseteq R_0$ and $R_1\subseteq \ggen{R_0}$.
Then the following is a sequence of Tietze transformations, with the types indicated, that transforms our starting presentation into one in with a smaller set of generators:
\[\langle X_0 \mid R_0 \rangle \overset{(T_1)}{\mapsto} \langle X_0 \mid R_0 \cup R_1 \rangle \overset{(T_1^{-1})}{\mapsto} \langle X_0 \mid R_0' \cup R_1 \rangle \overset{(T_2^{-1})}{\mapsto} \langle X_1 \mid R_1 \rangle.\]

Let $X_2 = X_1$, and let $R_2$ be the subset of $R_1$ comprising those relators that declare certain generators to have order two and those that label embedded circuits in $\Gamma(n)$ and which begin with one of the letters $a_{i, 1}, a_{i, 1}^{-1}, b_{i,1}$ or $b_{i, 1}^{-1}$; the relators omitted are those that correspond to paths that begin at vertices that were added in the subdivision of the original graph.   Since each word in $R_1 \setminus R_2$ is a cyclic conjugate of a word in $R_2$, the following is a Tietze transformation: 
\[\langle X_1 \mid R_1 \rangle \overset{(T_1^{-1})}{\mapsto} \langle X_2 \mid R_2 \rangle.\]

We introduce more new generators which correspond to labels on edges of $\Gamma$. Let $X_3 = X_2 \cup \Sigma$, and let 
\begin{multline*}R_3 = R_2 \cup \{a_{i, n+1}^{-1} a_{i, n}^{-1} \dots a_{i, 1}^{-1} a_{i} a_{i, 1} \dots a_{i, n} : i\in[1,m_1]\}\\ \cup \{b_{i, n+1}^{-1} b_{i, n}^{-1} \dots b_{i, 1}^{-1} b_i  b_{i, 2n+1}^{-1} \dots b_{i, n+2}^{-1} :  i\in[1,m_2]\}.\end{multline*}  Then the following is a Tietze transformation: 
\[\langle X_2 \mid R_2 \rangle \overset{(T_2)}{\mapsto} \langle X_3 \mid R_3 \rangle.\]

Let $X_4 = X_3$.  Let $R_3'$ be the set of relators in $R_2$ that express that certain generators have order two, let $R_3''$ be the set of relators in $R_3$ that label embedded circuits in $\Gamma(n)$ and begin with one of the letters $a_{i, 1}, a_{i, 1}^{-1}, b_{i_1}$ or $b_{i, 1}^{-1}$, and let $R_3'''$ be the set of relators in $R_3$ that were introduced in the previous paragraph.   Let $R_4''$ be the set of words obtained from $R_3''$ by replacing every occurrence of the factor 
$a_{i, 1} \cdots a_{i, n} a_{i, n+1} a_{i, n}^{-1} \cdots a_{i, 1}^{-1}$ by $a_i$, every occurrence of the factor $b_{i, 1} b_{i, 2} \cdots b_{i, 2n+1}$ by $b_i$, and every occurrence of the factor $b_{i, 2n+1}^{-1} \cdots b_{i, 2}^{-1} \dots b_{i, 1}^{-1}$ by $b_i^{-1}$.  Let $R_4 = R_3' \cup R_4'' \cup R_4'''$.  The following is a sequence of Tietze transformations:
\[\langle X_3 \mid R_3 \rangle \overset{(T_1)}{\mapsto} \langle X_4 \mid R_3 \cup R_4'' \rangle \overset{(T_1^{-1})}{\mapsto} \langle X_4 \mid R_4\rangle.\]

Finally, we remove the middle edge from each subdivided edge.   Let $X_5 = X_4 \setminus (\{a_{i, n+1} : i\in[1,m_1]\} \cup \{b_{i, n+1}: i\in[1, m_2]\})$ and let $R_5 = R_4 \setminus R_3'''$. We note that in $R_4$, the letter $a_{i, n+1}$ appears only in the relator $a_i^{-1} a_{i, 1} \dots a_{i, n} a_{i, n+1} a_{i, n}^{-1} \dots a_{i, 1}^{-1}$, the letter $b_{i, n+1}$ appears only in the relator $b_i^{-1} b_{i, 1} b_{i, 2} \dots b_{i, 2n+1}$, the letter $b_{i, n+1}^{-1}$ does not appear at all.   It follows that the following is a Tietze transformation:
\[\langle X_4 \mid R_4 \rangle \overset{(T_2^{-1})}{\mapsto} \langle X_5 \mid R_5 \rangle.\]

In summary, \[X_5 = \Sigma_1 \cup \Sigma_2 \cup \{a_{i, j} :  i \in[1, m_1],  j \in[1,] n] \} 
\cup \{b_{i, j} :  i \in[1, m_2],  j \in[1, 2n+1] \text{ and } j \neq n+1\},\]
and $R_5$ contains only words from $(\Sigma_1 \cup \Sigma_2 \cup \Sigma_2^{-1})^\ast$.   By \cref{lem:FreeProducts}, $\langle X_5 \mid R_5\rangle$ presents a group isomorphic to $\langle \Sigma_1 \cup \Sigma_2 \mid R_5\rangle \ast F( \{a_{i, j} :  i \in [1,m_1],  j \in[1, n] \} \cup \{b_{i, j} :  i \in[1, m_2], j \in [1,2n+1] \text{ and } j \neq n+1\})$.  Finally, we note that $\langle X_5 \mid R_5\rangle$ presents $G$.
\end{proof}

 We have now proved the following.

\setcounter{theoremx}{0}
\begin{theoremx}\label{thm:construction}
   Let $G$ be a group  with finite inverse-closed generating set $\Sigma$ and $n\in \mathbb{N}$. Then  the rewriting system $\nabla_n(G,\Sigma)=(\Sigma(n),T(n))$ 
    is  inverse-closed confluent and presents a finitely-generated group that is isomorphic to 
    $G\ast F_{n|\Sigma|}$. 
\end{theoremx}

\begin{proof}
    Follows from \cref{lem:nablaisICRS,prop:whatgroupisnablapresenting}.
\end{proof}

\cref{prop:whatgroupisnablapresenting}
 shows that we have not produced any new geodetic groups. On input of a plain group with geodetic generating set into $\mathcal{G}_n^{\text{gen}}$, we will output a plain group with geodetic generating set. However, we ask whether we have constructed any geodetic Cayley graphs different from those produced by the standard geodetic generating sets for plain groups (as seen in \cref{eg:standardgeodeticgensetforplaingroup}).

\begin{example}\label{rem:differentpresentations}
To indicate the variety of Cayley graphs for a particular isomorphism class of plain group, we list a number of geodetic generating sets for the free product of a cyclic group of order $5$ and a free group over a set $X$ of size $8$. Let $G=\mathcal{C}_5\ast F(X)$ with $\mathcal{C}_5=\langle a \mid a^5=1\rangle$ and $X=\{x_1,\dots,x_8\}$. We have the following geodetic generating sets of groups isomorphic to $G$ that result in non-isomorphic geodetic Cayley graphs:

\begin{multicols}{2}
\begin{itemize}
    \item[(i)] $\{a^{\pm1},x_1^{\pm 1},\dots,x_8^{\pm 1}\}$
    \item[(ii)] $\{a^{\pm 1}\}(1)\cup\{x_1^{\pm 1},\dots,x_6^{\pm 1}\}$
    \item[(iii)] $\{a^{\pm 1}\}(2)\cup\{x_1^{\pm 1},\dots,x_4^{\pm1}\}$
    \item[(iv)]  $\{a^{\pm 1}\}(3)\cup\{x_1^{\pm 1},x_2^{\pm1}\}$
    \item[(v)] $\{a^{\pm 1}\}(4)$
    \item[(vi)] $\{a^{\pm 1},a^{\pm 2},x_1^{\pm 1},\dots,x_8^{\pm 1}\}$
    \item[(vii)] $\{a^{\pm 1},a^{\pm 2}\}(1)\cup\{x_1^{\pm1},\dots,x_4^{\pm 1}\}$
      \item[(viii)] $\{a^{\pm 1},a^{\pm 2}\}(2)$
\end{itemize}\end{multicols}

Generating sets (i) and (vi)  correspond to standard  geodetic generating sets (see  
\cref{eg:standardgeodeticgensetforplaingroup}).
 The second generating set (ii) is obtained by taking the 1-subdivision of the odd-cycle Cayley graph for $C_5$ (which by Theorem gives a group isomomorphic to $C_5\ast F_2$ since $n=1$ and $|\Sigma|=1$) and  a rank 6 free factor (generated by $x_1\,dots, x_6$ and their inverses). The next three items are obtained similarly from the odd cycle.
 Items (vii) and (viii) are obtained by taking the 1-subdivision of the complete Cayley graph for $C_5$, adding the appropriate number of generators to make up the remaining free factor.

This is by no means an exhaustive list. Additional non-isomorphic Cayley graphs can readily be obtained using the method shown in \cref{eg:shortcutinfreegroups} to obtain alternative geodetic generating sets for free factors.
It is straightforward to prove these generating sets give distinct isomorphism classes of Cayley graphs. For example, the maximal cycle subgraphs contained in the Cayley graphs corresponding to (i)-(v) are $5,15,25,35$ and $45$ respectively.

In the next section we will investigate properties relating to isomorphisms of Cayley graphs produced by the construction. For example, \cref{prop:nablarespectsfreeproducts} explains why generating set (ii) and $\{a^{\pm1},x_1^{\pm1},x_2^{\pm1}\}(1)$ have isomorphic Cayley graphs and \cref{prop:iteratedsubdivision} explains why generating set (iii) and $[\{a^{\pm 1}\}(1)](1)$ have isomorphic Cayley graphs.
\end{example}

\section{Further properties of the construction}
\label{sec:propertiesoftheconstruction}

In this section we observe some simple facts about the construction, namely: that iterating the construction (applying subdivision again after an initial subdivision) has the same effect as subdividing with a larger number of subintervals; and applying the construction to a free product with a standard geodetic generating set is equivalent to applying the construction to each factor separately then forming the free product.

We start with a straightforward observation.

\begin{lemma}\label{lem:nis0doesnothing}
Let $G$ be a group with finite inverse-closed generating set $\Sigma$. Then $\text{Cay}(G,\Sigma)\cong \text{Cay}(\mathcal{G}^{\text{gen}}_0(G,\Sigma))$.
\end{lemma}

\begin{proof}
The bijection $\phi: \Sigma\to \Sigma(0)$ given by $\phi (x_i)\mapsto x_{i,1}$ induces a labelled graph isomorphism between the two Cayley graphs.
\end{proof}

\begin{proposition}\label{prop:iteratedsubdivision}
    Let $n,m\in \mathbb{N}$ and $G$ a group with finite inverse-closed generating set $\Sigma$. Then \[\text{Cay}(\mathcal{G}^{\text{gen}}_n\circ \mathcal{G}^{\text{gen}}_m (G,\Sigma))\cong \text{Cay}(\mathcal{G}^{\text{gen}}_k(G,\Sigma))\cong \text{Cay}(\mathcal{G}^{\text{gen}}_m\circ \mathcal{G}^{\text{gen}}_n(G,\Sigma)),\] where $k=2nm+n+m$.
\end{proposition}

\begin{proof}
Consider an edge of $\Gamma=\text{Cay}(G,\Sigma)$ and the process of subdividing and labelling done to this edge to construct $\Gamma_m$, $(\Gamma_m)_n$ and $\Gamma_k$. \cref{fig:edgecomparison} shows an edge in $\Gamma$ being labelled by an element of $\Sigma_2$ and $\Sigma_3$ (notation for edges labelled by an element of $\Sigma_1$ is essentially the same).

\begin{figure}[ht]
 \begin{subfigure}{0.9\textwidth}
 \begin{center}
      \begin{tikzpicture}

      \phantom{\node[circle,draw,color=gray,fill=gray,inner sep=1.3pt] (B01) at (11,0) {};\node[below right] at (B01) {\textcolor{red}{\tiny{$c_{i,2m+1,2n+1}$}}};} % to align four pics, this is the right-most part in (c)

      \node[circle, draw, fill=black,inner sep=1.3pt] (A) at (0,0) {};
        \node[circle, draw, fill=black,inner sep=1.3pt] (B) at (12,0) {};
        \node[above left] at (B) {\textcolor{blue}{$b_i$}};
         \node[below right] at (A) {\textcolor{red}{$c_{i}$}};

          \node[right] at (B) {$q$};
         \node[left] at (A) {{$p$}};
        \draw[thick,{Stealth[color=red]}-{Stealth[color=blue]}] (A) -- (B);
        
\end{tikzpicture}
\caption{An edge $\{p,q\}$ in $\Gamma$.
\label{fig:gamma}}
\end{center}
\end{subfigure}

\begin{subfigure}{0.9\textwidth}
      \begin{center}
      \begin{tikzpicture}

\phantom{\node[circle,draw,color=gray,fill=gray,inner sep=1.3pt] (B01) at (11,0) {};\node[below right] at (B01) {\textcolor{red}{\tiny{$c_{i,2m+1,2n+1}$}}};} % to align four pics, this is the right-most part in (c)

\node[circle, draw, fill=black,inner sep=1.3pt] (A) at (0,0) {};
\node[circle, draw, fill=black,inner sep=1.3pt] (B) at (12,0) {};

  \node[right] at (B) {$q$};
\node[left] at (A) {{$p$}};

\node[circle,draw,color=ForestGreen,fill=ForestGreen,inner sep=1.3pt] (A1) at (4,0) {};
 % \node[circle, draw,color=ForestGreen,fill=ForestGreen,inner sep=1.3pt] (A2) at (4,0) {};
\node[circle, draw,color=ForestGreen,fill=ForestGreen,inner sep=1.3pt] (B1) at (8,0) {};
% \node[circle, draw,color=ForestGreen,fill=ForestGreen,inner sep=1.3pt] (B1) at (7,0) {};

         \node[above left] at (B) {\textcolor{blue}{\tiny{$b_{i,2m+1}$}}};
         \node[below right] at (A) {\textcolor{red}{\tiny{$c_{{i},{1}}$}}};
        %    \node[above left] at (B1) {\textcolor{blue}{\tiny{$b_{i,2m}$}}};
        %  \node[below right] at (A1) {\textcolor{red}{\tiny{$c_{i,2}$}}};

\node[above left] at (A1) {\textcolor{blue}{\tiny{$b_{i,1}$}}};
% \node[above left] at (A2) {\textcolor{blue}{\tiny{$b_{i,2}$}}};
 \node[below right] at (B1) {\textcolor{red}{\tiny{$c_{i,2m+1}$}}};
%  \node[below right] at (B2) {\textcolor{red}{\tiny{$c_{i,2m}$}}};
         
        \draw[thick,{Stealth[color=red]}-{Stealth[color=blue]}] (A) -- (A1);
        % \draw[thick,{Stealth[color=red]}-{Stealth[color=blue]}] (A1) -- (A2);
          \draw[dashed,{Stealth[color=red]}-{Stealth[color=blue]}] (A1) -- (B1);
         %  \draw[thick,{Stealth[color=red]}-{Stealth[color=blue]}] (B2) -- (B1);
            \draw[thick,{Stealth[color=red]}-{Stealth[color=blue]}] (B1) -- (B);

\end{tikzpicture}

\caption{Edge $\{p,q\}$ subdivided in $\Gamma_m$. \label{fig:gammam}}
\end{center}
\end{subfigure}

   \begin{subfigure}{0.9\textwidth}
      \begin{center}
      \begin{tikzpicture}

\node[circle, draw, fill=black,inner sep=1.3pt] (A) at (0,0) {};
\node[circle, draw, fill=black,inner sep=1.3pt] (B) at (12,0) {};
\node[right] at (B) {$q$};
\node[left] at (A) {{$p$}};
\

 \node[circle, draw,color=ForestGreen,fill=ForestGreen,inner sep=1.3pt] (A1) at (4,0) {};
\node[circle, draw,color=ForestGreen,fill=ForestGreen,inner sep=1.3pt] (B1) at (8,0) {};

\node[circle,draw,color=gray,fill=gray,inner sep=1.3pt] (A01) at (1,0) {};
\node[circle,draw,color=gray,fill=gray,inner sep=1.3pt] (A11) at (3,0) {};

\node[circle,draw,color=gray,fill=gray,inner sep=1.3pt] (B11) at (9,0) {};
\node[circle,draw,color=gray,fill=gray,inner sep=1.3pt] (B01) at (11,0) {};

\node[above left] at (A1) {\textcolor{blue}{\tiny{$b_{i,1,2n+1}$}}};
\node[above left] at (A01) {\textcolor{blue}{\tiny{$b_{i,1,1}$}}};
\node[below right] at (A11) {\textcolor{red}{\tiny{$c_{i,1,2n+1}$}}};
\node[below right] at (A) {\textcolor{red}{\tiny{$c_{i,1,1}$}}};

\node[below right] at (B1) {\textcolor{red}{\tiny{$c_{i,2m+1,1}$}}};
\node[below right] at (B01) {\textcolor{red}{\tiny{$c_{i,2m+1,2n+1}$}}};
\node[above left] at (B11) {\textcolor{blue}{\tiny{$b_{i,2m+1,1}$}}};
\node[above left] at (B) {\textcolor{blue}{\tiny{$b_{i,2m+1,2n+1}$}}};

        \draw[thick,{Stealth[color=red]}-{Stealth[color=blue]}] (A) -- (A01);
\draw[dotted,{Stealth[color=red]}-{Stealth[color=blue]}] (A01) -- (A11);
 \draw[thick,{Stealth[color=red]}-{Stealth[color=blue]}] (A11) -- (A1);
 
  \draw[dashed,{Stealth[color=red]}-{Stealth[color=blue]}] (A1) -- (B1);
  
     \draw[thick,{Stealth[color=red]}-{Stealth[color=blue]}] (B1) -- (B11);
\draw[dotted,{Stealth[color=red]}-{Stealth[color=blue]}] (B11) -- (B01);
 \draw[thick,{Stealth[color=red]}-{Stealth[color=blue]}] (B01) -- (B);

\end{tikzpicture}

\caption{
Edge $\{p,q\}$ subdivided in$ (\Gamma_m)_n$. Note that the an edge labelled $b_{i,j}$ in $\Gamma_m$ is subdivided into $2n+1$ edges.
\label{fig:gammamn1}}
\end{center}
\end{subfigure}

   \begin{subfigure}{0.9\textwidth}
      \begin{center}
      \begin{tikzpicture}
\phantom{\node[below right] at (B01) {\textcolor{red}{\tiny{$c_{i,2m+1,2n+1}$}}};} % to align four pics, this is the right-most part in (c)

\node[circle, draw, fill=black,inner sep=1.3pt] (A) at (0,0) {};
\node[circle, draw, fill=black,inner sep=1.3pt] (B) at (12,0) {};
\node[right] at (B) {$q$};
\node[left] at (A) {{$p$}};

\node[circle,draw,color=gray,fill=gray,inner sep=1.3pt] (A1) at (1,0) {};
\node[circle,draw,color=gray,fill=gray,inner sep=1.3pt] (A2) at (2,0) {};

\node[circle,draw,color=gray,fill=gray,inner sep=1.3pt] (B2) at (10,0) {};
\node[circle,draw,color=gray,fill=gray,inner sep=1.3pt] (B1) at (11,0) {};

 \draw[thick,{Stealth[color=red]}-{Stealth[color=blue]}] (B1) -- (B);
  \draw[thick,{Stealth[color=red]}-{Stealth[color=blue]}] (B2) -- (B1);
    \draw[thick,{Stealth[color=red]}-{Stealth[color=blue]}] (A) -- (A1);
\draw[thick,{Stealth[color=red]}-{Stealth[color=blue]}] (A1) -- (A2);

\node[above left] at (B) {\textcolor{blue}{\tiny{$b'_{i,2k+1}$}}};
\node[above left] at (B1) {\textcolor{blue}{\tiny{$b'_{i,2k}$}}};
\node[below right] at (B1) {\textcolor{red}{\tiny{$c'_{2k+1}$}}};
\node[below right] at (B2) {\textcolor{red}{\tiny{$c'_{2k}$}}};

\node[above left] at (A1) {\textcolor{blue}{\tiny{$b'_{i,1}$}}};
\node[above left] at (A2) {\textcolor{blue}{\tiny{$b'_{i,2}$}}};
\node[below right] at (A) {\textcolor{red}{\tiny{$c'_{i,1}$}}};
\node[below right] at (A1) {\textcolor{red}{\tiny{$c'_{i,2}$}}};

%%Extra edges in middle
 \node[circle, draw,color=ForestGreen,fill=ForestGreen,inner sep=1.3pt] (A1) at (4,0) {};
\node[circle, draw,color=ForestGreen,fill=ForestGreen,inner sep=1.3pt] (B1) at (8,0) {};
\node[circle,draw,color=gray,fill=gray,inner sep=1.3pt] (A11) at (3,0) {};
\node[circle,draw,color=gray,fill=gray,inner sep=1.3pt] (B11) at (9,0) {};
 \draw[thick,{Stealth[color=red]}-{Stealth[color=blue]}] (A11) -- (A1);

   \draw[dashed,{Stealth[color=red]}-{Stealth[color=blue]}] (A1) -- (B1);
  
     \draw[thick,{Stealth[color=red]}-{Stealth[color=blue]}] (B1) -- (B11);
 \draw[dotted,{Stealth[color=red]}-{Stealth[color=blue]}] (A2) -- (A11);
 \draw[dotted,{Stealth[color=red]}-{Stealth[color=blue]}] (B11) -- (B2);
\node[above left] at (B11) {\textcolor{blue}{\tiny{$b'_{i,2m(2n+1)+1}$}}};
\node[above left] at (A1) {\textcolor{blue}{\tiny{$b'_{i,2n+1}$}}};
\node[below right] at (A11) {\textcolor{red}{\tiny{$c'_{i,2n+1}$}}};
\node[below right] at (B1) {\textcolor{red}{\tiny{$c'_{i,2m(2n+1)+1}$}}};

\end{tikzpicture}
\caption{
Edge $\{p,q\}$ subdivided in $\Gamma_k$.
\label{fig:gammamn2}}
\end{center}
\end{subfigure}
    \caption{Comparison of different subdivisions of an edge in $\Gamma$.}
    \label{fig:edgecomparison}
\end{figure}
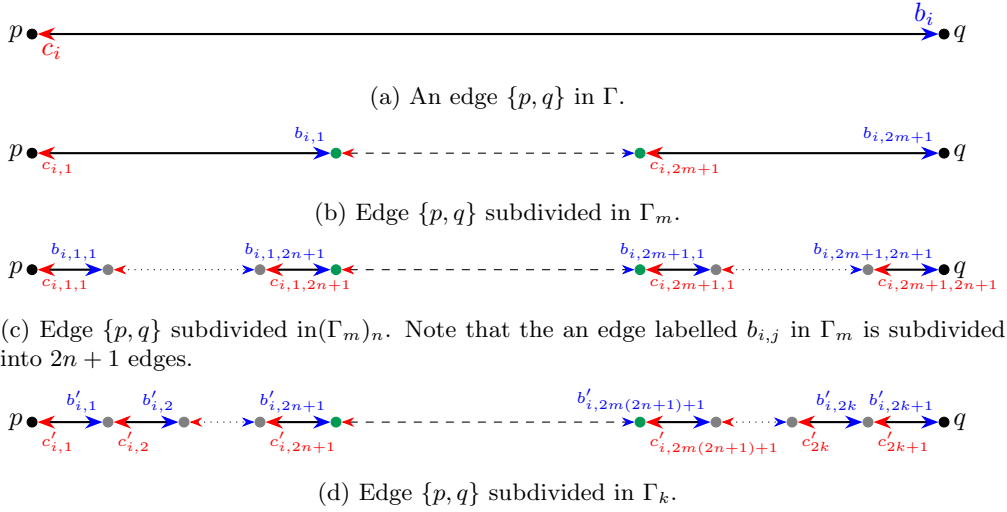

This leads to a
bijection $\phi: \Sigma(k)\to (\Sigma(m))(n)$ defined for $x=a,b,c$ as follows:
\[
\begin{matrix}
    x'_{i,1}\mapsto x_{i,1,1} & x'_{i,(2n+1)+1}\mapsto x_{i,2,1}  & ...&x'_{i,2m(2n+1)+1}\mapsto x_{i,2m+1,1} \\ 
     x'_{i,2}\mapsto x_{i,1,2}   &x'_{i,(2n+1)+2}\mapsto x_{i,2,2}  & ...  & x'_{i,2m(2n+1)+2}\mapsto x_{i,2m+1,2} \\ 
     \vdots &\vdots & &\vdots  \\
     x'_{i,2n+1}\mapsto x_{i,1,2n+1} &  x'_{i,2(2n+1)}\mapsto x_{i,2,2n+1}& ... & x'_{i,(2m+1)(2n+1)}\mapsto x_{i,2m+1,2n+1} \\
\end{matrix}
\]

One may verify that this bijection
 preserves rewriting rules, and hence induces a labelled graph isomorphism 
proving $\text{Cay}(\mathcal{G}^{\text{gen}}_n\circ \mathcal{G}^{\text{gen}}_m (G,\Sigma))\cong \text{Cay}(\mathcal{G}^{\text{gen}}_k(G,\Sigma))$.

A symmetric  argument shows that $\text{Cay}(\mathcal{G}^{\text{gen}}_m\circ \mathcal{G}^{\text{gen}}_n (G,\Sigma))\cong \text{Cay}(\mathcal{G}^{\text{gen}}_k(G,\Sigma))$.
\end{proof}

\begin{proposition}\label{prop:nablarespectsfreeproducts}
    Let $G=H\ast K$ with finite inverse-closed generating set $\Sigma_G=\Sigma_H\sqcup\Sigma_K$ such that $\Sigma_H$ is a generating set for $H$ and $\Sigma_K$ is a generating set for $K$. Then   \[\text{Cay}(\mathcal{G}^{\text{gen}}_n(G,\Sigma_G))\cong \text{Cay}(\mathcal{G}_n(H,\Sigma_H)\ast\mathcal{G}_n(K,\Sigma_K)),\Sigma_H(n)\sqcup\Sigma_K(n)).\]
\end{proposition}

\begin{proof}
    The construction in \cref{sec:TheConstruction} treats each generator of $\Sigma_G=\Sigma_H\sqcup \Sigma_K$ individually to construct the new alphabet, so $\Sigma_G(n)=\Sigma_H(n)\sqcup\Sigma_K(n)$. Furthermore, by inspection of the types of rules in the new rewriting system, we observe $T_G(n)=\Sigma_H(n) \sqcup \Sigma_K(n)$. The labelled graph isomorphism between the two Cayley graphs follows.
\end{proof}

Note that several notions of a free product of graphs exist (see for example \cite{FreeProdGraphsNewcatle}), which in the case of vertex-transitive graphs coincide. Using 
 this,
\cref{prop:nablarespectsfreeproducts} could be restated as   \[\text{Cay}(\mathcal{G}^{\text{gen}}_n(G,\Sigma_G))= \text{Cay}(\mathcal{G}^{\text{gen}}_n(H,\Sigma_H))\ast \text{Cay}(\mathcal{G}^{\text{gen}}_n(K,\Sigma_K)).\] 

\begin{example}
    Let $F(X)$ be a free group on $X=\{x_1,...,x_k\}$ and $F(Y)$ be a free group on $Y=\{y_1,...,y_{r(2n+1)}\}$. By \cref{prop:nablarespectsfreeproducts}, we have $\text{Cay}(\mathcal{G}_n^{\text{gen}}(F(X),X\sqcup X^{-1}))\cong \text{Cay}(F(Y),Y\sqcup Y^{-1})$. 
\end{example}

\section{Conclusion and outlook}

We have presented a way to construct new presentations for free products which yields new geodetic generating sets with interesting properties.   While the construction does not furnish examples which can resolve 
the conjectures of Madlener and Otto, Shapiro, and Federici respectively,
they give further insight into these difficult open problems.  
Since the Cayley graphs we construct are a simple generalisation of trees, they might also find application in parallel computing in place of (or combined with) trees and hypercubes.  

The construction is based on subdivision of edges in an existing geodetic graph to construct a new geodetic graph
due to Parathasarathy and Srinvivasan  \cite{parthasarathy1982some}. There are several other ways to construct new geodetic graphs from existing ones, as in for example
\cite{frasser2020geodetic,plesnik1977two,plesnik1984construction,stemple1979geodetic}. However these methods  only apply to particular types of graphs such as complete graphs and  Moore graphs, and they create vertices of different degrees, so cannot be vertex transitive, so will not yield Cayley graphs.

We conclude with two further remarks.
When $n\neq 0$, $\mathcal{G}_n(G,\Sigma)$ is a group with generating set $\Sigma(n)$ whose elements are either order $2$ or infinity. If we were to begin with a finite group $G$ of odd order with geodetic generating set $\Sigma$, then $\mathcal{G}_n(G,\Sigma)\cong G\ast F_{n|\Sigma|}$ would be a group with geodetic generating set $\Sigma(n)$ that contains no finite order elements, but still contains finite subgroups. Following this idea of removing finite order generators, it is worth noting that by \cite[Lemma 3.11]{Elder07052025}, any order $2$ element of a geodetic group $G$ must have a conjugate contained in any geodetic generating set for $G$, and so if the group has order $2$ elements a geodetic generating set must contain order $2$ elements. 

  Finally,  several authors have considered 
  the following generalisation of geodetic where $k\in \mathbb{N}_+$ \cite{elder2023kgeodeticwithjournal,bigeodetic}. If a graph has at most $k$ geodesics between each pair of vertices, then we call a graph \emph{$k$-geodetic}. A group is called $k$-geodetic if there exists an inverse-closed generating set such that Cayley graph is $k$-geodetic. See \cite{elder2023kgeodeticwithjournal} for some results on this topic. Suppose $k>1$ and that $G$ is group with $k$-geodetic generating set $\Sigma$ that is not a geodetic generating set. It is not hard to see that for all $n>1$ this $\text{Cay}(\mathcal{G}_n^{\text{gen}}(G,\Sigma))$ is not $k$-geodetic for any $k\in \mathbb{N}_+$. Note that, to date, the only known examples of $k$-geodetic groups that are not geodetic are finite.

\bibliographystyle{plain}
\bibliography{refs}

\begin{thebibliography}{10}

\bibitem{OttoBook}
Ronald~V. Book and Friedrich Otto.
\newblock {\em String-rewriting systems}.
\newblock Texts and Monographs in Computer Science. Springer-Verlag, New York,
  1993.

\bibitem{BridsonGeomofWordProb}
Martin~R. Bridson.
\newblock The geometry of the word problem.
\newblock In {\em Invitations to geometry and topology}, volume~7 of {\em Oxf.
  Grad. Texts Math.}, pages 29--91. Oxford Univ. Press, Oxford, 2002.

\bibitem{FreeProdGraphsNewcatle}
Max Carter, Stephan Tornier, and George Willis.
\newblock On free products of graphs.
\newblock {\em Australas. J. Combin.}, 78:154--176, 2020.

\bibitem{DAM}
Volker Diekert, Manfred Kufleitner, Gerhard Rosenberger, and Ulrich Hertrampf.
\newblock {\em Discrete algebraic methods}.
\newblock De Gruyter Textbook. De Gruyter, Berlin, 2016.
\newblock Arithmetic, cryptography, automata and groups.

\bibitem{EisenbergP19}
Andy Eisenberg and Adam Piggott.
\newblock Gilman's conjecture.
\newblock {\em J. Algebra}, 517:167--185, 2019.

\bibitem{ElderP22}
Murray Elder and Adam Piggott.
\newblock Rewriting systems, plain groups, and geodetic graphs.
\newblock {\em Theoret. Comput. Sci.}, 903:134--144, 2022.

\bibitem{ElderP23}
Murray Elder and Adam Piggott.
\newblock On groups presented by inverse-closed finite confluent
  length-reducing rewriting systems.
\newblock {\em J. Algebra}, 627:106--131, 2023.

\bibitem{Elder07052025}
Murray Elder, Adam Piggott, Florian Stober, Alexander Thumm, and Armin Weiß.
\newblock Finite groups with geodetic cayley graphs.
\newblock {\em Experimental Mathematics}, 0(0):1--24, 2025.

\bibitem{elder2023kgeodeticwithjournal}
Murray Elder, Adam Piggott, and Kane Townsend.
\newblock On {$k$}-geodetic graphs and groups.
\newblock {\em Internat. J. Algebra Comput.}, 33(6):1169--1182, 2023.

\bibitem{wrap108882}
Bruno Federici.
\newblock {\em Interactions between large-scale invariants in infinite graphs}.
\newblock PhD thesis, University of Warwick, 2017.

\bibitem{frasser2020geodetic}
Carlos Frasser and George Vostrov.
\newblock Geodetic {G}raphs {H}omeomorphic to a {G}iven {G}eodetic {G}raph.
\newblock {\em Int. J. Graph Theory Appl.}, 3:13--44, 12 2020.

\bibitem{Leighton1992}
F.~Thomson Leighton.
\newblock {\em Introduction to parallel algorithms and architectures}.
\newblock Morgan Kaufmann, San Mateo, CA, 1992.
\newblock Arrays, trees, hypercubes.

\bibitem{LyndonSchupp}
Roger~C. Lyndon and Paul~E. Schupp.
\newblock {\em Combinatorial group theory}.
\newblock Classics in Mathematics. Springer-Verlag, Berlin, 2001.
\newblock Reprint of the 1977 edition.

\bibitem{MadlenerOttoLengthReducing}
Klaus Madlener and Friedrich Otto.
\newblock Groups presented by certain classes of finite length-reducing
  string-rewriting systems.
\newblock In {\em Rewriting techniques and applications ({B}ordeaux, 1987)},
  volume 256 of {\em Lecture Notes in Comput. Sci.}, pages 133--144. Springer,
  Berlin, 1987.

\bibitem{parthasarathy1982some}
K.~R. Parthasarathy and N.~Srinivasan.
\newblock Some general constructions of geodetic blocks.
\newblock {\em J. Combin. Theory Ser. B}, 33(2):121--136, 1982.

\bibitem{AdamMonadic}
Adam Piggott.
\newblock On groups presented by monadic rewriting systems with generators of
  finite order.
\newblock {\em Bull. Aust. Math. Soc.}, 91(3):426--434, 2015.

\bibitem{plesnik1977two}
J{\'a}n Plesn{\'\i}k.
\newblock Two constructions of geodetic graphs.
\newblock {\em Math. Slovaca}, 27(1):65--71, 1977.

\bibitem{plesnik1984construction}
J{\'a}n Plesn{\'\i}k.
\newblock A construction of geodetic graphs based on pulling subgraphs
  homeomorphic to complete graphs.
\newblock {\em J. Combin. Theory Ser. B}, 36(3):284--297, 1984.

\bibitem{shapiro1997pascal}
Michael Shapiro.
\newblock Pascal's triangles in {A}belian and hyperbolic groups.
\newblock {\em J. Aust. Math. Soc.}, 63(2):281--288, 1997.

\bibitem{bigeodetic}
N.~Srinivasan, J.~Opatrn\'y, and V.~S. Alagar.
\newblock Bigeodetic graphs.
\newblock {\em Graphs Combin.}, 4(4):379--392, 1988.

\bibitem{stemple1979geodetic}
Joel~G. Stemple.
\newblock Geodetic graphs homeomorphic to a complete graph.
\newblock {\em An. New York Acad. Sci.}, 319(1):512--517, 1979.

\bibitem{YANG200773}
Jinn-Shyong Yang, Shyue-Ming Tang, Jou-Ming Chang, and Yue-Li Wang.
\newblock Parallel construction of optimal independent spanning trees on
  hypercubes.
\newblock {\em Parallel Computing}, 33(1):73--79, 2007.

\end{thebibliography}
\end{document}